\documentclass[11pt]{article}
\usepackage[tbtags]{amsmath}
\usepackage{amssymb}
\usepackage{amsthm}
\usepackage[misc]{ifsym}
\usepackage{cases}
\usepackage{cite}
\usepackage{mathrsfs}
\usepackage{xcolor}
\usepackage{graphicx}
\usepackage{float}
\usepackage{color}
\usepackage{hyperref}

\hypersetup{hypertex=true,
	colorlinks=true,
	linkcolor=red,
	anchorcolor=blue,
	citecolor=blue}

\numberwithin{equation}{section}
\normalsize

\title{\bf A Stochastic Linear-Quadratic Leader-Follower Differential Game with Elephant Memory
	\thanks{This work is financially supported by the National Key R\&D Program of China (2022YFA1006104), National Natural Science Foundations of China (12471419, 12271304), and Shandong Provincial Natural Science Foundations (ZR2024ZD35, ZR2022JQ01).}}
\author{\normalsize
	Xinpo Li\thanks{\it Research Center for Mathematics and Interdisciplinary Sciences, Shandong University, Qingdao 266237, P.R. China; and School of Mathematics, Shandong University, Jinan 250100, P.R. China, E-mail: lixinpo@mail.sdu.edu.cn},
   \ Jingtao Shi\thanks{\it Corresponding author, School of Mathematics, Shandong University, Jinan 250100, P.R. China, E-mail: shijingtao@sdu.edu.cn}}

\newtheorem{mythm}{Theorem}[section]
\newtheorem{mydef}{Definition}[section]
\newtheorem{mylem}{Lemma}[section]
\newtheorem{Remark}{Remark}[section]

\begin{document}
	
\maketitle

\noindent{\bf Abstract:}\quad
	This paper is concerned with a stochastic linear-quadratic leader-follower differential game with elephant memory. The model is general in that the state equation for both the leader and the follower includes the elephant memory of the state and the control, which are part of the diffusion term. Under certain assumptions, the state feedback representation of the open-loop Stackelberg strategy is derived by introducing two Riccati equations and a special matrix-valued equation. Finally, theoretical results are illustrated by means of an example concerning a dynamic advertising problem with elephant memory.
	
	\vspace{2mm}
	
	
	\vspace{2mm}
	
	
\section{Introduction}

	The Stackelberg game is a hierarchical, non-cooperative game (Ba\c{s}ar and Olsder \cite{BO95}), which can be traced back to the pioneering work \cite{Stack34} by von Stackelberg in 1934, and whose economic rationale is based on market phenomena in which some firms dominate others. Simultaneously, \cite{Stack34} is notable as the first to propose a hierarchical solution for the leader and the follower, now commonly renowned as the Stackelberg equilibrium. The leader and the follower, two players with asymmetric roles, constitute the Stackelberg game problem. To obtain a Stackelberg equilibrium, it is essential to follow a sequential approach consisting of two distinct stages. At the start, the follower adopts their optimal strategy to minimize his/her cost functional based on the choice announced by the leader. In the second phase, the leader, on the foundation of acquiring the optimal strategy for the follower, adapts the optimal strategy for himself/herself, such that the corresponding cost functional attains the minimal value.
	
	The Stackelberg game has widely recognized on account of its effectiveness in portraying real-world issues. Simann and Cruz \cite{SC73} conducted an early study on the properties of the Stackelberg solution in both static and dynamic non-zero-sum two-player games. The open-loop strategies for deterministic systems, inscribed by the introduced Riccati equations, have been further investigated in Pan and Yong \cite{PY91}, Ba\c{s}ar and Olsder \cite{BO95}, Freiling et al. \cite{FJL01}. Accordingly, stochastic systems have attracted extensive attention, a few of which are mentioned here. Castanon and Athans \cite{CM76} has previously studied this issue. Bagchi and Ba\c{s}ar \cite{BB81} analyzed the linear-quadratic Stackelberg differential game problem, focusing on the diffusion term of the state equation, which excludes state variables and control variables. Yong \cite{Y02} studied a linear-quadratic Stackelberg differential game problem with stochastic coefficients, where the control variables enter into the diffusion term of state equation and the condition that the coefficient matrix be positive definite has been eliminated. The corresponding open-loop optimal strategy is provided. Bensoussan et al. \cite{BCS15} investigated a linear-quadratic Stackelberg differential game problem featuring diffusion term without control, and proposed several solution concepts related to the information sets of players. Furthermore, the maximum principle under the adapted closed-loop memoryless information structure was established. Shi et al. \cite{SWX16, SWX17} and Zheng and Shi \cite{ZS22-2} studied Stackelberg differential games with asymmetric information.
	
	The dynamical systems discussed in the aforementioned literature rely solely on current values. However, in reality, past history can also influence certain stochastic systems. So far there have been extensive literature studying stochastic optimal control problems with time delay, see Mohammed \cite{Mohammed84, Mohammed98}, \O ksendal and Sulem \cite{OS00}, Chen and Wu \cite{CW10}, Meng and Shi \cite{MS21}, Meng et al. \cite{MSWZ23} and the references therein. Hence it is of great importance to consider time delay in the dynamics of Stackelberg differential games. There exist a few literatures in this topic and let us mention some of them. Harband \cite{H77} pioneered this direction, successfully demonstrating the existence of a monotonic solution to the nonlinear car-following equation. In a linear-quadratic differential game with time delay, the evolution of the game is characterized by coupled differential equations that involve both lumped and distributed parameter subsystems. Ishida and Shimemura \cite{IS83, IS87} established sufficient conditions for formulating a team-optimal closed-loop Stackelberg strategy and outlined the necessary and sufficient conditions for open-loop Stackelberg strategies, respectively. \O ksendal et al. \cite{OSU14} developed a continuous-time vertical contracting model to account for the fact that information is delayed in some market economies, furthermore explicitly formulated for equilibrium prices. Bensoussan et al. \cite{BCY15} investigated a linear-quadratic mean-field game between a leader and a group of followers under the Stackelberg game setting, where only the leader's state delay appears in the model. In a subsequent work, Bensoussan et al. \cite{BCLY17} further explored the issue of managing control delays. Xu and Zhang \cite{XZ16} focused on both difference and differential leader-follower games that involve time delays, and derived the necessary and sufficient conditions for solvability with the help of the new co-states introduced, which ensures the existence of a unique open-loop Stackelberg strategy for this problem. Additionally, they explicitly obtained the open-loop strategy in terms of decoupled and symmetric Riccati equations. Xu et al. \cite{XSZ18} studied a linear-quadratic leader-follower stochastic differential game, where the time delay only affects the leader's problem, and obtained an open-loop solution expressed in terms of the conditional expectation relating to several symmetric Riccati equations. Meng and Shi \cite{MS22} investigated a linear-quadratic stochastic Stackelberg differential game with time delay, where both state delay and control delay are present in the state equation, and they both contribute to the diffusion term. Li and Wang \cite{LW24} examined a linear-quadratic generalized Stackelberg game involving time-delay, motivated by the multi-scheme supply chain problem, which incorporates a multi-level hierarchical structure with delays.
	
	The aforementioned literature focus on stochastic control systems with point delay or average delay, where the delay time is typically treated as a fixed constant. In addition, Dahl et al. \cite{DMOR16} further investigated stochastic delayed systems with {\it noisy memory}, where the dynamics of the state depend on the past history of It\^o's stochastic integral. Subsequently, Agram and \O ksendal \cite{AO19-3} explored stochastic systems incorporating memory mean-field processes. However, the current state of certain systems, such as long-distance signaling and duopoly games in economic markets, is influenced not just by past values of the process over a specific interval, but by the entire history of the process, which is called {\it elephant memory} that is termed inspired by an American folklore that an elephant never forgets. Agram and \O ksendal \cite{AO19-2} and Feng et al. \cite{FGWX24} have investigated the optimal control problem for stochastic control systems with elephant memory consecutively. In view of the broader nature of elephant memory and the effectiveness of Stackelberg differential game in illustrating economic phenomena, we propose a comprehensive model that encompasses elephant memory and provides a state feedback representation for an open-loop solution to the stochastic linear-quadratic leader-follower differential game in this paper. To the best of our knowledge, this topic has not been studied in the literature. The contributions of this paper can be summarized as below.

(1) Our model is general. The state equation of both the leader and follower contains the elephant memory of the state and the control, and these memories are part of the diffusion term.
		
(2) To address the challenges posed by the elephant memory, we draw inspiration from \cite{AO19-2, FGWX24} and apply the theory of Fr\'echet derivative with its dual operators, which leads us to introduce the adjoint equation (\ref{m1-equation}) which is an {\it anticipated backward stochastic differential equation} (ABSDE for short) and the partially coupled {\it forward-backward stochastic differential equation} (FBSDE for short) (\ref{p2-equation}). The latter composed of a {\it stochastic differential equation} (SDE for short) of elephant memory alongside an ABSDE, thus, we derive Theorem \ref{thm3.1} and Theorem \ref{thm4.1}, respectively.

(3) Our conclusions are straightforward. The open-loop Stackelberg strategies are derived using two Riccati equations ((\ref{P1-equation}), (\ref{P2-equation})), along with a matrix-valued equation (\ref{L-equation}), where they are easier to solve compared to the Riccati equations (7) and (15) presented in \cite{MS22}.
		
(4) As an application, we explore a dynamic cooperative advertising problem with elephant memory in the marketing channel. Inspired by the realization that businessmen could determine the timing of advertisements based on the relationship between prior advertisements and sales of products in a real market economy.
	
	The paper is structured as follows. Section 2 formulates the stochastic Stackelberg differential game with elephant memory and introduces essential definition along with lemmas. The optimization problems for the follower and the leader are addressed in Sections 3 and 4, respectively. Section 5 is concerned with an examination of the practical application to the dynamic advertising problem. Finally, there are some conclusive remarks in Section 6.
	
\section{Problem statement and preliminaries}

	Let us first briefly introduce some concepts that will be utilized in this paper.
	
	Throughout this work, $\mathbb{R}^{n\times m}$ represents the Euclidean space of all $n\times m$ real matrices, while $\mathbb{S}^n$ denotes the space containing all $n\times n$ symmetric matrices. $\mathbb{R}^{n\times m}$ is abbreviated as $\mathbb{R}^n$ when $m=1$. The symbols $|\cdot|$ and $\langle\cdot,\cdot\rangle$ indicate the norm in $\mathbb{R}^n$ and the inner product, respectively. The superscript $^\top$ frequently appears in the rest of content, denoting the transpose of a vector or matrix. $I$ represents the identity matrix of appropriate dimensions.
	
	Let $T>0$ be a fixed finite real number, suppose $(\Omega,\mathcal{F},\{\mathcal{F}_t\}_{t\geq 0},\mathbb{P})$ is a complete filtered probability space with the filtration defined as $\{\mathcal{F}_t\}_{t\geq 0}:=\sigma\{W(s);0 \leq s\leq t\}$, where $\{W(t)\}_{t\geq 0}$ is a one-dimensional standard Brownian motion. Under the probability measure $\mathbb{P}$, $\mathbb{E}[\cdot]$ designates the expectation, consequently, the conditional expectation is symbolized as $\mathbb{E}^{\mathcal{F}_t}[\cdot]:=\mathbb{E}[\cdot|\mathcal{F}_t]$.
		
	Then, we present the following spaces. For an integer $p>0$, we define
	\begin{eqnarray*}\begin{aligned}
			&L^p([0,T];\mathbb{R}^{n\times n}):=\bigg\{\mathbb{R}^{n\times n}\mbox{-valued function }\phi(t);\ \int_0^T|\phi(t)|^pdt<\infty\bigg\},\\
			&L^{\infty}([0,T];\mathbb{R}^{n\times n}):=\bigg\{\mathbb{R}^{n\times n}\mbox{-valued function }\phi(t);\ \sup\limits_{0\leq t\leq T}|\phi(t)|dt<\infty\bigg\},\\
			&C([0,T];\mathbb{R}^n):=\bigg\{\mathbb{R}^n\mbox{-valued continuous function }\phi(t);\ \sup\limits_{0\leq t\leq T}|\phi(t)|<\infty\bigg\},\\
			&L^2_\mathcal{F}([0,T];\mathbb{R}^n):=\bigg\{\mathbb{R}^n\mbox{-valued }\mathcal{F}_t\mbox{-adapted process }\phi(t);\ \mathbb{E}\int_0^T|\phi(t)|^2dt<\infty\bigg\},\\
			&L^2_\mathcal{F}(\Omega;C([0,T];\mathbb{R}^n)):=\bigg\{\mathbb{R}^n\mbox{-valued }\mathcal{F}_t\mbox{-adapted process }\phi(t);\
			\mathbb{E}\Big[\sup\limits_{0\leq t\leq T}|\phi(t)|^2\Big]<\infty\bigg\},\\
		    &V^2_\mathcal{F}([0,t];\mathbb{R}^n):=\bigg\{\big\{\phi(s)\big\}_{0\le s\le t}\mid \mathbb{R}^n\mbox{-valued } \mathcal{F}_s\mbox{-adapted process }\phi(s),\ \mathbb{E}\int_0^t|\phi(s)|^2ds<\infty\bigg\},\\
	\end{aligned}\end{eqnarray*}
    \begin{eqnarray*}\begin{aligned}		
            &S^2_\mathcal{F}(\Omega;C([0,t];\mathbb{R}^n)):=\bigg\{\big\{\phi(s)\big\}_{0\le s\le t} \mid \mathbb{R}^n\mbox{-valued } \mathcal{F}_s\mbox{-adapted process }\phi(s),\\
            &\hspace{4cm}\mathbb{E}\Big[\sup\limits_{0\leq s\leq t}|\phi(s)|^2\Big]<\infty\bigg\}.	
	\end{aligned}\end{eqnarray*}

    Next, we elaborate on the issue we intend to examine in this work.

    The linear controlled system, which serves as the foundation for the Stackelberg stochastic differential game problem, is described by the following SDE that incorporates elephant memory:
	\begin{equation}\label{system equation}\left\{\begin{aligned}
			dx^{u^1,u^2}(t)=&\ \Big[ a^1(t)x^{u^1,u^2}(t)+a^2(t)x^{u^1,u^2}_t+b^1(t)u^1_t+b^2(t)u^2_t \Big]dt\\
			               &+\Big[ c^1(t)x^{u^1,u^2}(t)+c^2(t)x^{u^1,u^2}_t+d^1(t)u^1_t+d^2(t)u^2_t \Big]dW(t),\ t\in[0,T],\\
			 x^{u^1,u^2}(0)=&\ x_0,
	\end{aligned}\right.\end{equation}
	and the cost functionals for both the follower and the leader are defined as follows, respectively:
	\begin{equation}\begin{aligned}\label{cost}
			J^i(u^1(\cdot),u^2(\cdot)):=&\ \mathbb{E}\bigg\{\int_0^T \bigg[\Big\langle l^i(t) x^{u^1,u^2}(t), x^{u^1,u^2}(t)\Big\rangle+\Big\langle \bar{l}^i(t) x^{u^1,u^2}_t, x^{u^1,u^2}_t\Big\rangle\\
			&\qquad +\Big\langle r^i(t)u^i_t,u^i_t \Big\rangle\bigg]dt+\Big\langle g^ix^{u^1,u^2}(T),x^{u^1,u^2}(T) \Big\rangle\bigg\},\ i=1,2.
	\end{aligned}\end{equation}
	In (\ref{system equation}) and (\ref{cost}), $x^{u^1,u^2}(\cdot)\in\mathbb{R}^n$ denotes a state process with an initial value of $x_0$, $u^1(\cdot)\in\mathbb{R}^{k_1}$ and $u^2(\cdot)\in\mathbb{R}^{k_2}$ are employed to represent the follower's control and the leader's control, respectively. The complete trajectories of state process $x^{u^1,u^2}(\cdot)$ along with control processes $u^i(\cdot)$ until time $t$ are referred to as $x^{u^1,u^2}_t:=\big\{x^{u^1,u^2}(s)\big\}_{0\le s\le t}$ and $u^i_t:=\big\{u^i(s)\big\}_{0\le s\le t}$, respectively, for $i=1,2$. $a^i(\cdot),c^i(\cdot)\in L^{\infty}([0,T];\mathbb{R}^{n\times n})$, $b^i(\cdot),d^i(\cdot)\in L^{\infty}([0,T];\mathbb{R}^{n\times k_i})$, $l^i(\cdot), \bar{l}^i(\cdot)\in L^{\infty}([0,T];\mathbb{S}^n)$, $r^i(\cdot)\in L^{\infty}([0,T];\mathbb{S}^{k_i})$, $g^i\in \mathbb{S}^n$, for $i=1,2$.

	As shown by Theorem 2.2 in Section 5.2 of Mao \cite{MXR08}, for any $u^1(\cdot)\in L_{\mathcal{F}}^2([0,T];\mathbb{R}^{k_1})$, $u^2(\cdot)\in L_{\mathcal{F}}^2([0,T];\mathbb{R}^{k_2})$, there exists a unique solution $x^{u^1,u^2}(\cdot)\in L_{\mathcal{F}}^2(\Omega;C([0,T];\mathbb{R}^n))$ to the state equation (\ref{system equation}), with the cost functionals (\ref{cost}) being well-defined.

	The following is a description of our leader-follower stochastic linear-quadratic differential game with elephant memory. To begin with, for each potential choice $u^2(\cdot)\in\mathcal{U}^2[0,T]:=L_{\mathcal{F}}^2([0,T];\mathbb{R}^{k_2})$ selected by the leader, the follower aims to adopt a strategy $\bar{u}^1(\cdot)\in\mathcal{U}^1[0,T]:=L_{\mathcal{F}}^2([0,T];\mathbb{R}^{k_1})$ such that the associated cost functional $J^1(\bar{u}^1(\cdot),u^2(\cdot))$ is the minimum of $J^1(u^1(\cdot),\\u^2(\cdot))$ over $u^1(\cdot)\in\mathcal{U}^1[0,T]$. Subsequently, on the foundation of acquiring the optimal strategy $\bar{u}^1(\cdot)$ decided by the follower, the leader will then seek to formulate their own optimal strategy $\bar{u}^2(\cdot)$ similar to the follower. This means that the leader will adapt their strategy $\bar{u}^2(\cdot)\in\mathcal{U}^2[0,T]$ such that the corresponding cost functional $J^2(\bar{u}^1(\cdot),\bar{u}^2(\cdot))$ is the minimum of $J^2(\bar{u}^1(\cdot),u^2(\cdot))$ over $u^2(\cdot)\in\mathcal{U}^2[0,T]$. In summary, the follower is seeking to establish a mapping $\bar{\alpha}^1[\cdot]:\mathcal{U}^2[0,T]\rightarrow\mathcal{U}^1[0,T]$, while the leader is concerned with determining a control $\bar{u}^2(\cdot)\in\mathcal{U}^2[0,T]$ such that
	\begin{eqnarray*}\left\{\begin{aligned}
		&J^1(\bar{\alpha}^1(u^2(\cdot))[\cdot],u^2(\cdot))=\underset{u^1(\cdot)\in\,\mathcal{U}^1[0,T]}{\inf}J^1(u^1(\cdot),u^2(\cdot)),\ \forall u^2(\cdot)\in\mathcal{U}^2[0,T],\\
		&J^2(\bar{\alpha}^1(\bar{u}^2(\cdot))[\cdot],\bar{u}^2(\cdot))=\underset{u^2(\cdot)\in\,\mathcal{U}^2[0,T]}{\inf}J^2(\bar{\alpha}^1(u^2(\cdot))[\cdot],u^2(\cdot)).
	\end{aligned}\right.\end{eqnarray*}
	The pair $(\bar{\alpha}^1[\cdot],\bar{u}^2(\cdot))$ is designated an open-loop solution to the aforementioned Stackelberg differential game, the associated solution $\bar{x}(\cdot)\equiv x^{\bar{u}^1,\bar{u}^2}(\cdot)$ is known as an optimal trajectory. The primary objective of this paper is to identify and articulate a unique open-loop strategy.
	
	To overcome the challenges posed by the elephant memory, the definition of Fr\'echet derivative and its dual operators theory are rendered as follows.
	
	\begin{mydef}\label{Frechet derivative}
		Assume $\mathbb{X},\mathbb{Y}$ are two Banach spaces, with norms designated as $\parallel\cdot\parallel_{\mathbb{X}}$ and $\parallel\cdot\parallel_{\mathbb{Y}}$ respectively, we define the operator $ \beta:\mathbb{X}\to\mathbb{Y}$. For any $v\in\mathbb{X}$, if there exists a continuous linear operator $ \gamma:\mathbb{X}\to\mathbb{Y}$ such that
		\begin{equation*}
			\lim_{\substack{\parallel h\parallel\to0\\ h\in\mathbb{X}}}\frac{\parallel \beta(v+h)-\beta(v)-\gamma(h)\parallel_{\mathbb{Y}}}{\parallel h\parallel_{\mathbb{X}}}=0,
		\end{equation*}
		then the operator $\beta$ is claimed to be Fr\'echet differentiable at $v$, and $\bigtriangledown_v \beta$ is referred to as the gradient or the Fr\'echet derivative of $\beta$ at $v$.
	\end{mydef}
		Particularly, the statement $\bigtriangledown_v \beta=\beta$ holds true for all $v$ if $\beta$ is a linear operator. This conclusion is crucial in this paper.
		
	Now we demonstrate the dual operator, for a detailed proof, please refer to reference \cite{AO19-2}.
	\begin{mylem}\label{dual operator}
		Assume that the operator $\zeta(t)\equiv\zeta(t,\cdot):S^2_\mathcal{F}(\Omega;C([0,t];\mathbb{R}^n))\mapsto\mathbb{R}$ is uniformly bounded for $t\in[0,T]$, which is continuous and linear. Then the map
		\begin{equation*}
			x\mapsto\mathbb{E}\bigg[\int_{0}^{T}\big\langle \zeta(t),x_t\big\rangle dt\bigg],\ x\in S^2_\mathcal{F}(\Omega;C([0,t];\mathbb{R}^n))
		\end{equation*}
		is a continuous linear functional. In this perspective, with the assistance of Riesz representation theorem, there exists a unique process $\zeta^{\ast}\in S^2_\mathcal{F}(\Omega;C([0,t];\mathbb{R}^n))$ for any $\zeta\in S^2_\mathcal{F}(\Omega;C([0,t];\mathbb{R}^n))$ such that
		\begin{equation}\label{operator2}
			\mathbb{E}\bigg[\int_{0}^{T}\big\langle \zeta(t),x_t\big\rangle dt\bigg]=\mathbb{E}\bigg[\int_{0}^{T}\big\langle \zeta^{\ast}(t),x(t)\big\rangle dt\bigg].
		\end{equation}	
	\end{mylem}
	
	\begin{mylem}\label{dual operator two}
		If $\zeta(t,\cdot):S^2_\mathcal{F}(\Omega;C([0,t];\mathbb{R}^n))\mapsto S^2_\mathcal{F}(\Omega;C([0,t];\mathbb{R}^n))$ can be expressed by
		\begin{equation*}
			\zeta(t,x_t)=\big\langle \beta, x_t\big\rangle \kappa(t).
		\end{equation*}
		Then, (\ref{operator2}) is satisfied by
		\begin{equation}\label{relation}
			\zeta^{\ast}(t):=\big\langle \beta, {\kappa}^t\big\rangle,
		\end{equation}	
		where $\kappa^t:=\big\{\kappa (t+\theta)\big\}_{\theta \in[0,t]}$.
	\end{mylem}

\section{Optimization problem for the follower}

	We formulate the state feedback of the follower's optimal control based on the description of the Stackelberg differential game provided in Section 2. Essentially, this problem can be viewed as a stochastic linear-quadratic optimal control problem with elephant memory, regardless of the selection $u^2(\cdot)$ of the leader.

	Then, we present a detailed statement intended to resolve the optimization problem concerning the follower.

	Problem \textbf{(F-EMLQ)}: Seek a control $u^1(\cdot)$ of the follower such that (\ref{system equation}) is maintained when the cost functional $J^1(u^1(\cdot),u^2(\cdot))$ is minimized over $\mathcal{U}^1[0,T]$ for any $u^2(\cdot)\in\mathcal{U}^2[0,T]$.
	
	Firstly, we introduce the adjoint equation as follows:
	\begin{equation}\left\{\begin{aligned}\label{m1-equation}
			dm^1(t)=&\ \bigg\{-a^1(t)^\top m^1(t)-c^1(t)^\top n^1(t)-l^1(t)x^{\bar{u}^1,u^2}(t)-\mathbb{E}^{\mathcal{F}_t}\Big[\Big(a^2(t)^\top m^1(t)\\
			        &\ +c^2(t)^\top n^1(t)+\bar{l}^1(t)x^{\bar{u}^1,u^2}_t\Big)^\ast \Big]\bigg\}dt+n^1(t)dW(t),\ t\in[0,T],\\
			 m^1(T)=&\ g^1x^{\bar{u}^1,u^2}(T),\ m^1(t)=0,\ t\in(T,2T]; \ n^1(t)=0,\ t\in[T,2T].
	\end{aligned}\right.\end{equation}
	It is evident that (\ref{m1-equation}) constitutes an ABSDE, which is coupled with state equation (\ref{system equation}) (see \cite{AO19-2, FGWX24}). The ABSDE was initially proposed by Peng and Yang in \cite{PY09}, and then by Theorem 3.4 of \cite{FGWX24}, (\ref{m1-equation}) admits the unique solution $(m^1(\cdot),n^1(\cdot))\in L_{\mathcal{F}}^2(\Omega;C([0,T];\mathbb{R}^n))\times L_{\mathcal{F}}^2([0,T];\mathbb{R}^n)$.
	
	Subsequent to the establishment of the maximum principle in \cite{AO19-2, FGWX24}, the conclusion is outlined regarding the optimal control of the follower.
	
	\begin{mythm}\label{thm3.1}
		For any $u^2(\cdot)\in\mathcal{U}^2[0,T]$, suppose $\bar{u}^1(\cdot)$ is the optimal control of the follower, the corresponding optimal state trajectory indicated as $x^{\bar{u}^1,u^2}(\cdot)$ for the Problem (\textbf{F-EMLQ}), and $\bar{l}^1(t)= r^1(t)=0$ holds true for $t\in [T,2T]$, then the result follows that the optimal control $\bar{u}^1(\cdot)$ satisfies
		\begin{equation}\label{u1-optimal}
			\mathbb{E}^{\mathcal{F}_t}\Big\{\big[b^1(t)^\top m^1(t)+d^1(t)^\top n^1(t)+ r^1(t)\bar{u}^1_t\big]^\ast \Big\}=0,\ a.e.\ t\in[0,T],\ \mathbb{P}\mbox{-}a.s..
		\end{equation}
	\end{mythm}
	\begin{proof}
		For any $v^1(\cdot)\in\mathcal{U}^1[0,T]$, define $u^{1\rho}(\cdot):=\bar{u}^1(\cdot)+\rho(v^1(\cdot)-\bar{u}^1(\cdot))$, $\rho\in[0,1]$. Assume the state trajectories $x^{v^1,u^2}(\cdot)$ and $x^{u^{1\rho},u^2}(\cdot)$ correspond to $v^1(\cdot),u^{1\rho}(\cdot)$, respectively. This yields
		\begin{equation*}\begin{aligned}
				&J^1(u^{1\rho}(\cdot),u^2(\cdot))-J^1(\bar{u}^1(\cdot),u^2(\cdot))\\
			    =\ &2\rho\mathbb{E}\bigg\{\int_0^T\bigg[\Big\langle l^1(t)x^{\bar{u}^1,u^2}(t),x^{v^1,u^2}(t)-x^{\bar{u}^1,u^2}(t)\Big\rangle
                +\Big\langle \bar{l}^1(t)x^{\bar{u}^1,u^2}_t,x^{v^1,u^2}_t-x^{\bar{u}^1,u^2}_t\Big\rangle\\
				&\quad\quad\quad+\Big\langle  r^1(t)\bar{u}^1_t,v^1_t-\bar{u}^1_t\Big\rangle\bigg]dt+\Big\langle g^1x^{\bar{u}^1,u^2}(T),x^{v^1,u^2}(T)-x^{\bar{u}^1,u^2}(T)\Big\rangle\bigg\}\\
		\end{aligned}\end{equation*}
        \begin{equation*}\begin{aligned}		
             &+\rho^2\mathbb{E}\bigg\{\int_0^T\bigg[\Big\langle l^1(t)\big(x^{v^1,u^2}(t)-x^{\bar{u}^1,u^2}(t)\big),x^{v^1,u^2}(t)-x^{\bar{u}^1,u^2}(t)\Big\rangle\\
             &\qquad\qquad\quad +\Big\langle \bar{l}^1(t)\big(x^{v^1,u^2}_t-x^{\bar{u}^1,u^2}_t\big),x^{v^1,u^2}_t-x^{\bar{u}^1,u^2}_t\Big\rangle+\Big\langle  r^1(t)\big(v^1_t-\bar{u}^1_t\big),v^1_t-\bar{u}^1_t\Big\rangle\bigg]dt\\
             &\qquad\quad +\Big\langle g^1\big(x^{v^1,u^2}(T)-x^{\bar{u}^1,u^2}(T)\big),x^{v^1,u^2}(T)-x^{\bar{u}^1,u^2}(T)\Big\rangle\bigg\}.
		\end{aligned}\end{equation*}
	    Next, by applying It\^o's formula to $\big\langle m^1(\cdot),x^{v^1,u^2}(\cdot)-x^{\bar{u}^1,u^2}(\cdot)\big\rangle$ and substituting into the above, we obtain the following deduction
		\begin{equation}\begin{aligned}\label{rho}
				&J^1(u^{1\rho}(\cdot),u^2(\cdot))-J^1(\bar{u}^1(\cdot),u^2(\cdot))\\
			  =\ &2\rho\mathbb{E}\int_0^T\bigg[\Big\langle \mathbb{E}^{\mathcal{F}_t}\big[b^1(t)^\top m^1(t)+d^1(t)^\top n^1(t)+ r^1(t)\bar{u}^1_t\big]^\ast,v^1(t)-\bar{u}^1(t)\Big\rangle\bigg]dt+o(\rho).
		\end{aligned}\end{equation}
	    The division of both sides of (\ref{rho}) by $\rho$, we gain
		\begin{equation*}\begin{aligned}
				0&\geq\frac{1}{\rho}\big[J^1(u^{1\rho}(\cdot),u^2(\cdot))-J^1(\bar{u}^1(\cdot),u^2(\cdot))\big]\\
				&=2\mathbb{E}\int_0^T\bigg[\Big\langle \mathbb{E}^{\mathcal{F}_t}\big[b^1(t)^{\top}m^1(t)+d^1(t)^{\top}n^1(t)+ r^1(t)\bar{u}^1_t\big]^{\ast },v^1(t)-\bar{u}^1(t)\Big\rangle\bigg]dt
               +\frac{o(\rho)}{\rho}.
		\end{aligned}\end{equation*}
		At last, letting $\rho\rightarrow 0$, thus we accomplish this proof in the fact of the arbitrariness to $v^1(\cdot)$.
	\end{proof}
	
	Secondly, we attempt to construct the feedback expression of $\bar{u}^1_{\cdot}$ from (\ref{u1-optimal}). Due to certain technical limitations, the following assumptions have been imposed on the coefficients of (\ref{system equation}) and (\ref{cost}) to achieve this particular objective:
	\begin{equation*}
		\textbf{(A1)}\left\{\begin{aligned}
			&c^1(t)^\top \Pi ^1(t)d^1(t)+\Pi ^1(t)b^1(t)=0,\ a.e.\ t\in[0,T],\\
			&c^1(t)^\top \Pi ^1(t)c^2(t)+\Pi ^1(t)a^2(t)=0,\ a.e.\ t\in[0,T],\\
			&c^2(t)^\top \Pi ^1(t)c^2(t)+\bar{l}^1(t)-c^2(t)^\top \Pi ^1(t)d^1(t)[\Xi^1(t)]^{-1}d^1(t)^\top \Pi ^1(t)c^2(t)=0,\ a.e.\ t\in[0,T],
		\end{aligned}\right.
	\end{equation*}
	where \ $\Xi^1(t):=r^1(t)+d^1(t)^\top \Pi ^1(t)d^1(t)>0,\ \forall\ t\in[0,T]$, and $\Pi^1(\cdot)$ is the solution to the Riccati equation as follows:
	\begin{equation}\left\{\begin{aligned}\label{P1-equation}
			\dot{\Pi}^1(t)=&-\Pi ^1(t)a^1(t)-a^1(t)^\top \Pi ^1(t)-c^1(t)^\top \Pi ^1(t)c^1(t)-l^1(t),\ t\in[0,T],\\
			      \Pi ^1(T)=&\ g^1,\ \Pi ^1(t)=0,\ t\in(T,2T].
		\end{aligned}\right.\end{equation}
	
	We introduce another ABSDE satisfied by $(\eta^1(\cdot),\bar{\eta}^1(\cdot))$:
	\begin{equation}\left\{\begin{aligned}\label{eta1-equation}
			d\eta^1(t)=&\bigg\{-a^1(t)^\top \eta^1(t)-c^1(t)^\top \bar{\eta}^1(t)-\mathbb{E}^{\mathcal{F}_t}\Big[\big(a^2(t)^\top \eta^1(t)+c^2(t)^\top \bar{\eta}^1(t)\big)^\ast\Big]\\
			&+\big[c^2(t)^\top \Pi ^1(t)d^1(t)\big]^\ast \big[\Xi^{1\ast}(t)\big]^{-1}\Big\{\mathbb{E}^{\mathcal{F}_t}\big[\big(b^1(t)^\top \eta^1(t)+d^1(t)^\top \bar{\eta}^1(t)\\
			&+d^1(t)^\top \Pi ^1(t)d^2(t)u^2_t\big)^\ast \big]\Big\}-\big[\Pi ^1(t)b^2(t)+c^1(t)^\top \Pi ^1(t)d^2(t)\big]u^2_t\\
			&-\mathbb{E}^{\mathcal{F}_t}\Big\{ \big[c^2(t)^\top \Pi ^1(t)d^2(t)u^2_t\big]^\ast \Big\}\bigg\}dt+\bar{\eta}^1(t)dW(t),\ t\in[0,T],\\
			\eta^1(t)=&\ \bar{\eta}^1(t)=0,\ t\in[T,2T].
		\end{aligned}\right.\end{equation}
	
	\begin{Remark}\label{rem3.1}
		As illustrated in equation (\ref{P1-equation}), this Riccati equation is less complex than equation (7) in \cite{MS22}, which is essentially a linear {\it ordinary differential equation} (ODE for short). Additionally, combined with the fact that the coefficients are bounded, (\ref{P1-equation}) admits the unique solution $\Pi^1(\cdot)\in C([0,T];\mathbb{S}^n)$. 	
	\end{Remark}
	
	\begin{Remark}\label{rem3.2}
		There exists a unique process pair $(\eta^1(\cdot),\bar{\eta}^1(\cdot))\in L_{\mathcal{F}}^2(\Omega;C([0,T];\mathbb{R}^n))\times L_{\mathcal{F}}^2([0,T];\\\mathbb{R}^n)$ such that the ABSDE (\ref{eta1-equation}) is satisfied for any $u^2(\cdot)\in\mathcal{U}^2[0,T]$, the conclusion derived by the boundedness of the coefficients (see \cite{FGWX24}).
	\end{Remark}
	
	Thirdly, the necessary and sufficient conditions of Problem \textbf{(F-EMLQ)} for solvability are derived.
	\begin{mythm}\label{thm3.2}
		Assuming that \textbf{(A1)} holds, let $\bar{l}^1(t)=r^1(t)=0$ for $t\in [T,2T]$. Suppose $\Pi^1(\cdot)$ as well as $(\eta^1(\cdot),\bar{\eta}^1(\cdot))$ is the unique solution of (\ref{P1-equation}) and (\ref{eta1-equation}), respectively.
		Then for any $u^2(\cdot)\in\mathcal{U}^2[0,T]$, it is concluded that the problem \textbf{(F-EMLQ)} is solvable and the optimal control of elephant memory $\bar{u}^1_{\cdot}$ takes the following state feedback form:
		\begin{equation}\begin{aligned}\label{u1-feedback}
				\bar{u}^1_t=&-\big[\Xi^1(t)\big]^{-1}\Big\{d^1(t)^\top \Pi ^1(t)c^2(t)x^{\bar{u}^1,u^2}_t+b^1(t)^\top \eta^1(t)+d^1(t)^\top \bar{\eta}^1(t)\\
				&\hspace{2cm}+d^1(t)^\top \Pi ^1(t)d^2(t)u^2_t\Big\},\ a.e.\  t\in[0,T],\ \mathbb{P}\mbox{-}a.s..
		\end{aligned}\end{equation}
		Furthermore, the optimal cost can be described as follows:
		\begin{equation}\begin{aligned}\label{follower-optimal-cost}
				&\underset{u^1(\cdot)\in\,\mathcal{U}^1[0,T]}{\inf}J^1(u^1(\cdot),u^2(\cdot))=J^1(\bar{u}^1(\cdot),u^2(\cdot))\\
				&=\mathbb{E}\bigg\{\big\langle \Pi^1(0)x_0+2\eta^1(0),x_0\big\rangle+\int_0^T\bigg[\big\langle \Pi ^1(t)d^2(t)u^2_t,d^2(t)u^2_t\big\rangle\\
				&\quad\quad+2\big\langle b^2(t)^\top \eta^1(t)+d^2(t)^\top \bar{\eta}^1(t),u^2_t\big\rangle-\Big|\big[\Xi^1(t)\big]^{-\frac{1}{2}}\big[b^1(t)^\top \eta^1(t)\\
				&\quad\quad+d^1(t)^\top \bar{\eta}^1(t)+d^1(t)^\top \Pi ^1(t)d^2(t)u^2_t\big]\Big|^2\bigg]dt\bigg\}.
		\end{aligned}\end{equation}
	\end{mythm}
	\begin{proof}
		For any $u^1(\cdot)\in\mathcal{U}^1[0,T]$, the application of It\^o's formula to $\langle \Pi^1(\cdot)x^{u^1,u^2}(\cdot),x^{u^1,u^2}(\cdot)\rangle$ and $2\langle\eta^1(\cdot),x^{u^1,u^2}(\cdot)\rangle$, in conjunction with \textbf{(A1)} and a series of calculations, leads to the following deduction
		\begin{equation}\begin{aligned}\label{eq3.9}
			   &J^1(u^1(\cdot),u^2(\cdot))\\
		     =&\ \mathbb{E}\bigg\{\langle \Pi^1(0)x_0+2\eta^1(0),x_0\rangle+\int_0^T\bigg[\big\langle \Pi ^1(t)d^2(t)u^2_t,d^2(t)u^2_t\big\rangle+2\big\langle b^2(t)^\top \eta^1(t)\\
			   &\quad+d^2(t)^\top \bar{\eta}^1(t),u^2_t\big\rangle-\Big|\big[\Xi^1(t)\big]^{-\frac{1}{2}}\big[b^1(t)^\top \eta^1(t)+d^1(t)^\top \bar{\eta}^1(t)+d^1(t)^\top \Pi ^1(t)d^2(t)u^2_t\big]\Big|^2\\
				&\quad +\Xi^1(t)\Big|u^1_t+\big[\Xi^1(t)\big]^{-1}\big\{d^1(t)^\top \Pi ^1(t)c^2(t)x^{\bar{u}^1,u^2}_t+b^1(t)^\top \eta^1(t)+d^1(t)^\top \bar{\eta}^1(t)\\
				&\quad+d^1(t)^\top \Pi^1(t)d^2(t)u^2_t\big\}\Big|^2\bigg]dt\bigg\},
		\end{aligned}\end{equation}
		which indicates that (\ref{u1-feedback}) represents the optimal control and (\ref{follower-optimal-cost}) is established.	
	\end{proof}
	
	\begin{Remark}\label{rem3.3}
		In the process of solving the follower's problem, we address the challenges posed by the elephant memory with the help of Lemmas \ref{dual operator}, \ref{dual operator two} and derive the result (\ref{u1-optimal}) in Theorem \ref{thm3.1}, which is a complex process defined over the interval $[0,T]$ and serves as a further generalization of the conclusions drawn from the corresponding problem that only involves point delays in literatures. Moreover, Lemmas \ref{dual operator} and \ref{dual operator two} are instrumental in obtaining the follower's state feedback of optimal control, and the conclusion (\ref{follower-optimal-cost}) in Theorem (\ref{thm3.2}) contains the elephant memory of the state.
	\end{Remark}
	
	\begin{Remark}\label{rem3.4}
		The assumption \textbf{(A1)} is imposed to remove the difficulties presented by the elephant memory. It is evident that this assumption is more robust than the one presented in \cite{MS22}. The consideration has been confined to Fr\'echet derivative of continuous linear operators and dual theory, and it is acknowledged that the higher order theories are more intricate, we intend to revisit these in future research. Moreover, this is attributable to the fact that the feedback of $\bar{u}^1_t$ at time $t$ in this game problem will comprise the state of elephant memory $x^{\bar{u}^1,u^2}_t:=\big\{x^{\bar{u}^1,u^2}(s)\big\}_{0\le s\le t}$ and the state $[x^{\bar{u}^1,u^2}_t]^\ast=\big\{x^{\bar{u}^1,u^2}(s+\theta)\big\}_{0\le \theta \le s,\ 0\le s\le t}$. In particular, assumption $c^1(\cdot)^\top \Pi^1(\cdot)d^1(\cdot)+\Pi^1(\cdot)b^1(\cdot)=0$ and  $c^2(\cdot)^\top \Pi^1(\cdot)c^2(\cdot)+\bar{l}^1(\cdot)-c^2(\cdot)^\top \Pi^1(\cdot)d^1(\cdot)[\Xi^1(\cdot)]^{-1}d^1(\cdot)^\top \Pi^1(\cdot)\\c^2(\cdot)=0$ overcome the challenges posed by $[x^{\bar{u}^1,u^2}_t]^\ast$, while $c^1(\cdot)^{\top}\Pi^1(\cdot)c^2(\cdot)+\Pi^1(\cdot)a^2(\cdot)=0$ successfully addresses the obstacles presented by $x^{\bar{u}^1,u^2}_t$. Therefore, the state feedback (\ref{u1-feedback}) is deduced. Conversely, the following state feedback is attained:
		\begin{equation*}
			\bar{u}^1_t=\{\cdots\}x^{\bar{u}^1,u^2}_t+\{\cdots\}
			    \mathbb{E}^{\mathcal{F}_t }\big\{[x^{\bar{u}^1,u^2}_t]^{\ast}\big\},
		\end{equation*}
		which is more complex than (\ref{u1-feedback}). Currently, an effective technique to address this issue has not yet been developed as we know. In the future, we will examine this problem with weaker assumptions than \textbf{(A1)} or without \textbf{(A1)}.
	\end{Remark}

	\section{Optimization problem for the leader}
	
	The optimal control problem for the leader will be considered in this section, which is treated as a stochastic optimal control problem. The state equation for this problem is constituted by a SDE with elephant memory and an ABSDE.
	
	At the start, we present the state equation for the leader by substituting (\ref{u1-feedback}) into (\ref{system equation}):
	\begin{equation}\label{leader state}\left\{\begin{aligned}
			dx^{\bar{u}^1,u^2}(t)=&\Big\{\bar{a}^1(t)x^{\bar{u}^1,u^2}(t)+\bar{a}^2(t)x^{\bar{u}^1,u^2}_t+\bar{b}(t)u^2_t+\bar{f}(t)\eta^1(t)+\bar{h}(t)\bar{\eta}^1(t)\Big\}dt\\
			&+\Big\{\bar{c}^1(t)x^{\bar{u}^1,u^2}(t)+\bar{c}^2(t)x^{\bar{u}^1,u^2}_t+\bar{d}(t)u^2_t+\bar{k}(t)\eta^1(t)+\bar{p}(t)\bar{\eta}^1(t)\Big\}dW(t),\\
			d\eta^1(t)=&\bigg\{-\bar{a}^1(t)^\top \eta^1(t)-\bar{c}^1(t)^\top \bar{\eta}^1(t)-\mathbb{E}^{\mathcal{F}_t}\Big[\big(\bar{a}^2(t)^\top \eta^1(t)+\bar{c}^2(t)^\top \bar{\eta}^1(t)\big)^\ast \Big]\\
			&+\mathbb{E}^{\mathcal{F}_t}\Big[\big(\bar{q}^1(t)u^2_t\big)^\ast \Big]+\bar{q}^2(t)u^2_t\bigg\}dt+\bar{\eta}^1(t)dW(t),\ t\in[0,T],\\
			x^{\bar{u}^1,u^2}(0)=&\ x_0,\  \eta^1(t)=\bar{\eta}^1(t)=0,\ t\in[T,2T],
		\end{aligned}\right.\end{equation}
	where
	\begin{equation*}\begin{aligned}
			&\bar{a}^1(t):= a^1(t),\quad\quad \bar{a}^2(t):=a^2(t)-b^1(t)\big[\Xi^1(t)\big]^{-1}d^1(t)^\top \Pi^1(t)c^2(t), \\
			&\bar{b}(t):= b^2(t)-b^1(t)\big[\Xi^1(t)\big]^{-1}d^1(t)^\top \Pi^1(t)d^2(t),\\
			&\bar{c}^1(t):=c^1(t), \quad\quad \bar{c}^2(t):=c^2(t)-d^1(t)\big[\Xi^1(t)\big]^{-1}d^1(t)^\top \Pi^1(t)c^2(t),\\
			&\bar{d}(t):=d^2(t)-d^1(t)\big[\Xi^1(t)\big]^{-1}d^1(t)^\top \Pi^1(t)d^2(t),\\
			&\bar{f}(t):=-b^1(t)\big[\Xi^1(t)\big]^{-1}b^1(t)^\top,\quad\quad \bar{h}(t):=-b^1(t)\big[\Xi^1(t)\big]^{-1}d^1(t)^\top,\\
			&\bar{k}(t):=-d^1(t)\big[\Xi^1(t)\big]^{-1}b^1(t)^\top,\quad\quad \bar{p}(t):=-d^1(t)\big[\Xi^1(t)\big]^{-1}d^1(t)^\top,\\
			&\bar{q}^1(t):=-c^2(t)^\top \Pi^1(t)\big\{I-d^1(t)\big[\Xi^1(t)\big]^{-1}d^1(t)^\top \Pi^1(t)\big\}d^2(t),\\
			&\bar{q}^2(t):=-\Pi^1(t)b^2(t)-c^1(t)^\top \Pi^1(t)d^2(t).
	\end{aligned}\end{equation*}
	
	Now the stochastic optimal control problem is expressed as follows for the leader.
	
	Problem \textbf{(L-EMLQ)}: Seek a control $u^2(\cdot)$ of the leader such that $(\ref{leader state})$ is maintained while the cost functional $J^2(\bar{u}^1(\cdot),u^2(\cdot))$ is minimized over $\mathcal{U}^2[0,T]$.
	
	Next, the adjoint equation is introduced to solve the Problem \textbf{(L-EMLQ)} as follows:
	\begin{equation}\left\{\begin{aligned}\label{p2-equation}
			d\xi(t)=&\bigg\{\bar{a}^1(t)\xi(t)+\bar{a}^2(t)\xi_t+\bar{f}(t)^\top m^2(t)+\bar{k}(t)^\top n^2(t)\bigg\}dt\\
			        &+\bigg\{\bar{c}^1(t)\xi(t)+\bar{c}^2(t)\xi_t+\bar{h}(t)^\top m^2(t)+\bar{p}(t)^\top n^2(t)\bigg\}dW(t),\\
			dm^2(t)=&\bigg\{-\bar{a}^1(t)^\top m^2(t)-\bar{c}^1(t)^\top n^2(t)-l^2(t)\bar{x}(t)-\mathbb{E}^{\mathcal{F}_t}\Big[\big(\bar{a}^2(t)^\top m^2(t)\\
			        &+\bar{c}^2(t)^\top n^2(t)+\bar{l}^2(t)\bar{x}_t\big)^\ast \Big]\bigg\}dt+n^2(t)dW(t),\ t\in[0,T],\\
			\xi(0)=&\ 0,\ m^2(T)=g^2\bar{x}(T),\ m^2(t)=0,\ t\in(T,2T];\ n^2(t)=0,\ t\in[T,2T],
		\end{aligned}\right.\end{equation}
	where $\bar{x}(\cdot)\equiv x^{\bar{u}^1,\bar{u}^2}(\cdot)$. The partially coupled equation (\ref{p2-equation}) is the adjoint equation which is derived from a SDE with elephant memory and an ABSDE. The first and second equation in (\ref{p2-equation}) admit unique solution $\xi(\cdot)\in L_{\mathcal{F}}^2(\Omega;C([0,T];\mathbb{R}^n))$ and $\big(m^2(\cdot),n^2(\cdot)\big)\in L_{\mathcal{F}}^2(\Omega;C([0,T];\mathbb{R}^n))\times L_{\mathcal{F}}^2([0,T];\mathbb{R}^n)$ (see \cite{MXR08,FGWX24}), respectively.
	
	Subsequently, similar to Theorem (\ref{thm3.1}), the following conclusion is obtained.
	\begin{mythm}\label{thm4.1}
		Suppose $\bar{u}^2(\cdot)$ is the optimal control of the leader, the corresponding optimal state trajectory indicated as $\bar{x}(\cdot)$ for the Problem (\textbf{L-EMLQ}), and $\bar{l}^2(t)= r^2(t)=0$ holds true for $t\in [T,2T]$, then the result follows that the optimal control $\bar{u}^2(\cdot)$ satisfies
		\begin{equation}\begin{aligned}\label{u2-optimal}
				\mathbb{E}^{\mathcal{F}_t}\Big[\big(\bar{b}(t)^\top m^2(t)+\bar{d}(t)^\top n^2(t)+ r^2(t)\bar{u}^2_t-\bar{q}^2(t)^\top \xi(t)-\bar{q}^1(t)^\top \xi_t\big)^{\ast}\Big]=0,&\\
                \ a.e.\ t\in[0,T],\ \mathbb{P}\mbox{-}a.s..&
		\end{aligned}\end{equation}
	\end{mythm}
	
	 Afterwards, we are going to investigate the feedback expression of $\bar{u}^2_{\cdot}$.
	 In addition to \textbf{(A1)}, several other assumptions must be made regarding the coefficients of (\ref{system equation}) and (\ref{cost}):
	\begin{equation*}
		\textbf{(A2)}\left\{\begin{aligned}
			&c^1(t)^\top \Pi^2(t)d^1(t)+\Pi^2(t)b^1(t)=0,\quad a.e.\ t\in[0,T],\\
			&c^1(t)^\top \Pi^2(t)c^2(t)+\Pi^2(t)a^2(t)=0,\quad a.e.\ t\in[0,T],\\
			&c^1(t)^\top \Pi^2(t)d^2(t)+\Pi^2(t)b^2(t)=0,\quad a.e.\ t\in[0,T],\\
			&\bar{c}^2(t)^\top \Pi^2(t)\bar{c}^2(t)+\bar{l}^2(t)-\bar{c}^2(t)^\top \Pi^2(t)\bar{d}(t)[\Xi^2(t)]^{-1}\bar{d}(t)^\top \Pi^2(t)\bar{c}^2(t)=0,\ a.e.\ t\in[0,T],
		\end{aligned}\right.
	\end{equation*}
	where $\Xi^2(t):=r^2(t)+\bar{d}(t)^\top \Pi^2(t)\bar{d}(t)>0,\ \forall t\in[0,T]$, and $\Pi^2(\cdot)$ is the solution to the Riccati equation as follows:
	\begin{equation}\left\{\begin{aligned}\label{P2-equation}
			\dot{\Pi}^2(t)=&-\Pi^2(t)\bar{a}^1(t)-\bar{a}^1(t)^\top \Pi^2(t)-\bar{c}^1(t)^\top \Pi^2(t)\bar{c}^1(t)-l^2(t),\ t\in[0,T],\\
			      \Pi^2(T)=&\ g^2,\ \Pi^2(t)=0,\ t\in(T,2T].
		\end{aligned}\right.\end{equation}
	
	Select $(\eta^2(\cdot),\bar{\eta}^2(\cdot))$ that satisfy the following ABSDE:
	\begin{equation}\left\{\begin{aligned}\label{eta2-equation}
			d\eta^2(t)=&\bigg\{-\bar{a}^1(t)^\top \eta^2(t)-\bar{c}^1(t)^\top \bar{\eta}^2(t)+\big\{\big[\bar{c}^2(t)^\top \Pi^2(t)\bar{d}(t)\big]^\ast \big[\Xi^{2\ast}(t)\big]^{-1}\big[\bar{b}(t)^\top \big]^\ast\\
			            &-\big[\bar{a}^2(t)^\top \big]^\ast \big\}\mathbb{E}^{\mathcal{F}_t}\big[\eta^{2\ast}(t)\big]+\big\{\big[\bar{c}^2(t)^\top \Pi^2(t)\bar{d}(t)\big]^\ast \big[\Xi^{2\ast}(t)\big]^{-1}\big[\bar{d}(t)^\top \big]^{\ast}\\
			            &-\big[\bar{c}^2(t)^\top \big]^\ast \big\}\mathbb{E}^{\mathcal{F}_t}\big[\bar{\eta}^{2\ast}(t)\big]-\big[\bar{c}^2(t)^\top \Pi^2(t)\bar{d}(t)\big]^\ast \big[\Xi^{2\ast}(t)\big]^{-1}\mathbb{E}^{\mathcal{F}_t}\big\{\big[\bar{q}^1(t)^\top \xi_t\\
			            &+\bar{q}^2(t)^\top \xi(t)\big]^\ast \big\}+\big\{\big[\bar{c}^2(t)^\top \Pi^2(t)\bar{d}(t)\big]^\ast \big[\Xi^{2\ast}(t)\big]^{-1}\big[\bar{d}(t)^\top \Pi^2(t)\bar{k}(t)\big]^\ast\\
			            &-\big[\bar{c}^2(t)^\top \Pi^2(t)\bar{k}(t)\big]^\ast \big\}\mathbb{E}^{\mathcal{F}_t}\big[\eta^{1\ast}(t)\big]+\big\{\big[\bar{c}^2(t)^\top \Pi^2(t)\bar{d}(t)\big]^\ast \big[\Xi^{2\ast}(t)\big]^{-1}\\
			            &\times \big[\bar{d}(t)^\top \Pi^2(t)\bar{p}(t)\big]^\ast-\big[\bar{c}^2(t)^\top \Pi^2(t)\bar{p}(t)\big]^\ast \big\}\mathbb{E}^{\mathcal{F}_t}\big[\bar{\eta}^{1\ast}(t)\big]\bigg\}dt\\
			            &+\bar{\eta}^2(t)dW(t),\ t\in[0,T],\\
			\eta^2(t)=&\ \bar{\eta}^2(t)=0,\ t\in[T,2T].
		\end{aligned}\right.\end{equation}
	
	\begin{Remark}\label{rem4.1}
		Similar to (\ref{P1-equation}), the Riccati equation (\ref{P2-equation}) can be solved combined with the boundedness of the coefficients as it is essentially a linear ODE. In the meanwhile, the ABSDE (\ref{eta2-equation}) admits a unique solution $(\eta^2(\cdot),\bar{\eta}^2(\cdot))\in L_{\mathcal{F}}^2(\Omega;C([0,T];\mathbb{R}^n))\times L_{\mathcal{F}}^2([0,T];\mathbb{R}^n)$(see \cite{FGWX24}).
	\end{Remark}

	It is straightforward to confirm that
	\begin{equation}\left\{\begin{aligned}\label{m2 n2}
			m^2(t)=&\ \Pi^2(t)\bar{x}(t)+\eta^2(t),\ t\in[0,T],\\
			n^2(t)=&\ \Pi^2(t)\bar{c}^1(t)\bar{x}(t)+\Pi^2(t)\bar{c}^2(t)\bar{x}_t+\Pi^2(t)\bar{d}(t)\bar{u}^2_t\\
			       &+\Pi^2(t)\bar{k}(t)\eta^1(t)+\Pi^2(t)\bar{p}(t)\bar{\eta}^1(t)+\bar{\eta}^2(t),\ t\in[0,T].
		\end{aligned}\right.\end{equation}
	Besides, with several calculations, (\ref{u2-optimal}) becomes
	\begin{equation}\begin{aligned}\label{u2-optimal-2}
			\bar{u}^2_t=-\big[\Xi^2(t)\big]^{-1}\bigg\{&\bar{d}(t)^{\top}\Pi^2(t)\bar{c}^2(t)\bar{x}_t+\bar{b}(t)^{\top}\eta^2(t)+\bar{d}(t)^{\top}\bar{\eta}^2(t)\\
			                                               &+\bar{d}(t)^{\top}\Pi^2(t)\bar{k}(t)\eta^1(t)+\bar{d}(t)^{\top}\Pi^2(t)\bar{p}(t)\bar{\eta}^1(t)\\
                                                         &-\bar{q}^1(t)^{\top}\xi_t-\bar{q}^2(t)^{\top}\xi(t)\bigg\},\ a.e.\ t\in[0,T],\ \mathbb{P}\mbox{-}a.s..
	\end{aligned}\end{equation}

	It is not satisfactory to observe the representation of the optimal control (\ref{u2-optimal-2}), as determining $\xi(\cdot)$ necessitates solving the adjoint equation (\ref{p2-equation}), which is difficult to deal with due to the lack of information on $\bar{x}(T)$. It is anticipated that a feedback expression for the optimal control of elephant memory $u^2_{\cdot}$ will be obtained, analogous to (\ref{u1-feedback}). In order to realize this objective, the forward variables and the backward variables obtained in the optimization of the follower and the leader, respectively, will be stacked. The details of this process will be described in the remainder of this section.
	
	Denote
   \begin{equation*}
      \phi(t):= \begin{bmatrix}
				\xi(t)\\
				\bar{x}(t)
			\end{bmatrix},\qquad
          \psi(t):= \begin{bmatrix}\eta^1(t)\\
				\eta^2(t)
			\end{bmatrix},\qquad
          \bar{\psi}(t):= \begin{bmatrix}\bar{\eta}^1(t)\\
				\bar{\eta}^2(t)
			\end{bmatrix},\\
   \end{equation*}
	\begin{eqnarray*}\begin{aligned}
			&A^1(t):= \begin{bmatrix}
				\bar{a}^1(t) & \bar{f}(t)^\top\Pi^2(t)+\bar{k}(t)^\top\Pi^2(t)\bar{c}^1(t)\\
				0 & \bar{a}^1(t)
			\end{bmatrix},\qquad
			A^2(t):= \begin{bmatrix}
				\bar{a}^2(t) & \bar{k}(t)^\top\Pi^2(t)\bar{c}^2(t)\\
				0 & \bar{a}^2(t)
			\end{bmatrix},\\
			&\bar{A}^1(t):= \begin{bmatrix}
				\bar{c}^1(t) & \bar{h}(t)^\top\Pi^2(t)+\bar{p}(t)^\top\Pi^2(t)\bar{c}^1(t)\\
				0 & \bar{c}^1(t)
			\end{bmatrix},\qquad
            \bar{A}^2(t):= \begin{bmatrix}
				\bar{c}^2(t) & \bar{p}(t)^\top\Pi^2(t)\bar{c}^2(t)\\
				0 & \bar{c}^2(t)
			\end{bmatrix},\\
            &B(t):= \begin{bmatrix}
				\bar{k}(t)^\top\Pi^2(t)\bar{k}(t) & \bar{f}(t)^\top\\
				\bar{f}(t) & 0
			\end{bmatrix},\qquad
            C(t):= \begin{bmatrix}
				\bar{k}(t)^\top\Pi^2(t)\bar{p}(t) & \bar{k}(t)^\top\\
				\bar{h}(t) & 0
			\end{bmatrix},\\
			&\bar{C}(t):= \begin{bmatrix}
				\bar{p}(t)^\top\Pi^2(t)\bar{p}(t) & \bar{p}(t)^\top\\
				\bar{p}(t) & 0
			\end{bmatrix},\qquad
            D(t):= \begin{bmatrix}
				\bar{k}(t)^\top\Pi^2(t)\bar{d}(t)\\
				\bar{b}(t)
			\end{bmatrix},\\
			&\bar{D}(t):= \begin{bmatrix}
				\bar{p}(t)^\top\Pi^2(t)\bar{d}(t)\\
				\bar{d}(t)
			\end{bmatrix},\qquad
            G^1(t):= \begin{bmatrix}
				\bar{q}^1(t)\\
				-\bar{c}^2(t)^\top\Pi^2(t)\bar{d}(t)
			\end{bmatrix},\\
			&G^2(t):= \begin{bmatrix}
				\bar{q}^2(t)\\
				0
			\end{bmatrix},\qquad
            H(t):= \begin{bmatrix}
				0 & 0\\
				0 & -\bar{c}^2(t)^\top\Pi^2(t)\bar{d}(t)[\Xi^2(t)]^{-1}\bar{d}(t)^\top\Pi^2(t)\bar{c}^2(t)
			\end{bmatrix},\\
	\end{aligned}\end{eqnarray*}
	then we gain
	\begin{equation}\left\{\begin{aligned}\label{total-state-equation}
			d\phi(t)=&\bigg\{A^1(t)\phi(t)+A^2(t)\phi_t+B(t)\psi(t)+C(t)\bar{\psi}(t)+D(t)\bar{u}^2_t\bigg\}dt\\
			         &+\bigg\{\bar{A}^1(t)\phi(t)+\bar{A}^2(t)\phi_t+C(t)^\top\psi(t)+\bar{C}(t)\bar{\psi}(t)+\bar{D}(t)\bar{u}^2_t\bigg\}dW(t),\\
			d\psi(t)=&\bigg\{\mathbb{E}^{\mathcal{F}_t}\big\{\big[H(t)\phi_t\big]^\ast\big\}-A^1(t)^\top\psi(t)
-\mathbb{E}^{\mathcal{F}_t}\big\{\big[A^2(t)^\top\psi(t)\big]^\ast\big\}-\bar{A}^1(t)^\top\bar{\psi}(t)\\
			         &-\mathbb{E}^{\mathcal{F}_t}\big\{\big[\bar{A}^2(t)^\top\bar{\psi}(t)\big]^\ast\big\}
+\mathbb{E}^{\mathcal{F}_t}\big\{\big[G^1(t)\bar{u}^2_t\big]^\ast\big\}+G^2(t)\bar{u}^2_t\bigg\}dt\\
			         &+\bar{\psi}(t)dW(t),\ t\in[0,T],\\
			 \phi(0)=&\ (0,x_0^\top)^\top,\ \psi(t)=\bar{\psi}(t)=(0,0)^\top,\ t\in[T,2T],
		\end{aligned}\right.\end{equation}
	and (\ref{u2-optimal-2}) can be simplified as
	\begin{equation}\begin{aligned}\label{u2-optimal-3}
			\bar{u}^2_t=-\big[\Xi^2(t)\big]^{-1}\Big\{-G^1(t)^\top\phi_t+D(t)^\top\psi(t)+\bar{D}(t)^\top\bar{\psi}(t)-G^2(t)^\top\phi(t)\Big\},&\\
			             \quad a.e.\ t\in[0,T],\ \mathbb{P}\mbox{-}a.s..&
	\end{aligned}\end{equation}
	
	Establishing a nonhomogeneous relationship between $\phi(\cdot)$ and $\psi(\cdot)$ is crucial for state feedback to the optimal control of elephant memory $\bar{u}^2_{\cdot}$. Next, based on \textbf{(A1)}, \textbf{(A2)}, we also impose additional conditions on the coefficients of (\ref{system equation}) and (\ref{cost}):
	\begin{equation*}
		\textbf{(A3)}\left\{\begin{aligned}
			&I-d^1(t)\big[\Xi^1(t)\big]^{-1}d^1(t)^\top\Pi^1(t)=0,\quad a.e.\ t\in[0,T],\\
			&a^2(t)-b^1(t)\big[\Xi^1(t)\big]^{-1}d^1(t)^\top\Pi^1(t)c^2(t)=0,\quad a.e.\ t\in[0,T].
		\end{aligned}\right.
	\end{equation*}	
	Hence
	\begin{equation*}\left\{\begin{aligned}
			&\bar{c}^2(t)\equiv\bar{d}(t)\equiv\bar{q}^1(t)\equiv0,\quad t\in[0,T],\\
			&\bar{A^2}(t) \equiv \bar{D}(t)\equiv G^1(t) \equiv   H(t) \equiv0,\quad t\in[0,T],\\
			&A^2(t)=\begin{bmatrix}
				\bar{a}^2(t) & 0\\
				0 & \bar{a}^2(t)
			\end{bmatrix},\quad \bar{D}(t)=\begin{bmatrix}
				0\\
				\bar{d}^2(t)
			\end{bmatrix},\quad t\in[0,T].
		\end{aligned}\right.\end{equation*}
	
	Denote
	\begin{equation*}\begin{aligned}
			\Xi^3(t):=&\ r^2(t),\\
			\Omega^1(t):=&\ B(t)\Gamma(t)+C(t)\big[I-\Gamma(t)\bar{C}(t)\big]^{-1}\Gamma(t)\big[\bar{A}^1(t)+C(t)^\top\Gamma(t)\big]\\
			&+D(t)\big[\Xi^3(t)\big]^{-1}\big[G^2(t)^{\top}-D(t)^\top\Gamma(t)\big],\\
			\Omega^2(t):=& -B(t)-C(t)\big[I-\Gamma(t)\bar{C}(t)\big]^{-1}\Gamma(t)C(t)^{\top}+D(t)\big[\Xi^3(t)\big]^{-1}D(t)^{\top},\\
			\Omega^3(t):=&\ \big[G^2(t)-\Gamma(t)D(t)\big]\big[\Xi^3(t)\big]^{-1}\big[G^2(t)^{\top}-D(t)^{\top}\Gamma(t)\big]-\big[\bar{A}^1(t)+C(t)^{\top}\Gamma(t)\big]^{\top}\\
			&\times\big[I-\Gamma(t)\bar{C}(t)\big]^{-1}\Gamma(t)\big[\bar{A}^1(t)+C(t)^{\top}\Gamma(t)\big]-\Gamma(t)B(t)\Gamma(t),
	\end{aligned}\end{equation*}
	where $\Gamma(\cdot)$ satisfies the matrix-valued equation as follows:
	\begin{equation}\left\{\begin{aligned}\label{L-equation}
			\dot{\Gamma}(t)&=-\Gamma(t)A^1(t)-A^1(t)^\top\Gamma(t)+\Lambda(t,t+\theta)I_{[0,T]}(t+\theta),\ \theta\in[0,t], \ t\in[0,T],\\
			\Gamma(T)&=0,
		\end{aligned}\right.\end{equation}
	and $\Lambda(\cdot,\cdot+\theta)$ is the solution to the following equation:
	\begin{equation}\left\{\begin{aligned}\label{Pi-equation}
			-\frac{\partial\Lambda(t,s)}{\partial t}=&\ \Lambda(t,s)A^1(t)+A^1(t)^\top\Lambda(t,s),\ s-\theta \leq t\leq s,\\
			\Lambda(t,t)=&\left(\int_t^{(t+\theta)\wedge T}\Lambda(t,s)ds\right)\Omega^1(t)+\Omega^1(t)^\top\int_t^{(t+\theta)\wedge T}\Lambda(t,s)ds\\
			&+\left(\int_t^{(t+\theta)\wedge T}\Lambda(t,s)ds\right)\Omega^2(t)\int_t^{(t+\theta)\wedge T}\Lambda(t,s)ds+\Omega^3(t).
		\end{aligned}\right.\end{equation}
	Obviously $\Lambda(\cdot,\cdot+\theta)$ is symmetric, hence so is $\Gamma(\cdot)$.
	
	\begin{Remark}\label{rem4.3}
	The solvability of these matrix-valued equations is rather challenging. We will assume that their solvability holds true in this paper.
	\end{Remark}
	
	Furthermore, we develop the nonhomogeneous relationship between $\psi(\cdot)$ and $\phi(\cdot)$, as illustrated below.
	
	\begin{mylem}\label{lem4.1}
		Assume that there are unique solutions to both the matrix equation $(\ref{L-equation})$ and $(\ref{Pi-equation})$, $[I-\Gamma(\cdot)\bar{C}(\cdot)]^{-1}$ exits, then it holds that
		\begin{equation}\label{relationship}
			\psi(t)=\Gamma(t)\phi(t)-\int_t^{(t+\theta)\wedge T}\Lambda(t,s)\mathbb{E}^{\mathcal{F}_t}[\phi(t)]ds,\ \theta\in[0,t],\ t\in[0,T].
		\end{equation}
	\end{mylem}
	\begin{proof}
		Define
		\begin{equation}\label{relationship--}
			\Upsilon(t)\triangleq \psi(t)-\Gamma(t)\phi(t),\quad t\in[0,T],
		\end{equation}
		and suppose $d\Upsilon(t)=\Upsilon^1(t)dt+\Upsilon^2(t)dW(t)$.
		The result is as follows which is obtained by applying It\^o's formula to (\ref{relationship--}):
		\begin{equation*}\begin{aligned}
				d\Upsilon(t)=&\Big\{-\big[A^1(t)^\top+\Gamma(t)B(t)\big]\psi(t)-\big[\bar{A}^1(t)^\top+\Gamma(t)C(t)\big]\bar{\psi}(t)\\
				&\ +\big[G^2(t)-\Gamma(t)D(t)\big]\bar{u}^2_t+\big[A^1(t)^\top\Gamma(t)-\Lambda(t,t+\theta)I_{[0,T]}(t+\theta)\big]\phi(t)\Big\}dt\\
				&+\Big\{\bar{\psi}(t)-\Gamma(t)\bar{A}^1(t)\phi(t)-\Gamma(t)C(t)^\top\psi(t)-\Gamma(t)\bar{C}(t)\bar{\psi}(t)\Big\}dW(t),\\
                &\hspace{6cm}  \theta\in[0,t], \ t\in[0,T],	
		\end{aligned}\end{equation*}
		then we gain
		\begin{equation*}
				\bar{\psi}(t)=\Upsilon^2(t)+\Gamma(t)\bar{A}^1(t)\phi(t)+\Gamma(t)C(t)^\top\psi(t)+\Gamma(t)\bar{C}(t)\bar{\psi}(t),\quad t\in[0,T].
		\end{equation*}
		As a result, we deduce
		\begin{equation}\begin{aligned}\label{relationship---}
				\bar{\psi}(t)=&\big[I-\Gamma(t)\bar{C}(t)\big]^{-1}\Big\{\Upsilon^2(t)+\Gamma(t)\bar{A}^1(t)\phi(t)+\Gamma(t)C(t)^\top\psi(t)\Big\},\quad t\in[0,T].
		\end{aligned}\end{equation}
		Moreover, substituting (\ref{relationship--}) into (\ref{u2-optimal-3}) yields
		\begin{equation}\begin{aligned}\label{u2---}
				\bar{u}^2_t=&-\big[\Xi^3(t)\big]^{-1}\Big\{\big[D(t)^{\top}\Gamma(t)-G^2(t)^{\top}\big]\phi(t)+D(t)^{\top}\Upsilon(t)\Big\},\ a.e.\ t\in[0,T],\ \mathbb{P}\mbox{-}a.s..
		\end{aligned}\end{equation}
		Therefore we derive
		\begin{equation}\begin{aligned}\label{Gamma1}
				\Upsilon^1(t)=&-\Gamma(t)B(t)\Gamma(t)\phi(t)-\Lambda(t,t+\theta)I_{[0,T]}(t+\theta)\phi(t)+\big[\bar{A}^1(t)+\Gamma(t)C(t)\big]\bar{\psi}(t)\\
             &+\big[G^2(t)-\Gamma(t)D(t)\big]\bar{u}^2_t-\big[A^1(t)^{\top}+\Gamma(t)B(t)\big]\Upsilon(t),\ \theta\in[0,t], \ t\in[0,T].
		\end{aligned}\end{equation}
		Taking (\ref{relationship---}) and (\ref{u2---}) into (\ref{Gamma1}) gives rise to the following conclusion:
		\begin{equation}\begin{aligned}\label{Gamma--}
				\Upsilon^1(t)=&\Big\{ \big[G^2(t)-\Gamma(t)D(t)\big]\big[\Xi^3(t)\big]^{-1}\big[G^2(t)^{\top}-D(t)^{\top}\Gamma(t)\big]-\Gamma(t)B(t)\Gamma(t)\\
             &-\big[\bar{A}^1(t)^{\top}+\Gamma(t)C(t)\big]\big[I-\Gamma(t)\bar{C}(t)\big]^{-1}\Gamma(t)\big[\bar{A}^1(t)+C(t)^{\top}\Gamma(t)\big]\Big\}\phi(t)\\
				&-\Lambda(t,t+\theta)I_{[0,T]}(t+\theta)\phi(t)+\Big\{-A^1(t)^{\top}-\Gamma(t)B(t)-\big[G^2(t)-\Gamma(t)D(t)\big]\\
				&\times \big[\Xi^3(t)\big]^{-1}D(t)^{\top}-\big[\bar{A}^1(t)^{\top}+\Gamma(t)C(t)\big]\big[I-\Gamma(t)\bar{C}(t)\big]^{-1}\Gamma(t)C(t)^{\top}\Big\}\Upsilon(t)\\
				&-\big[\bar{A}^1(t)^{\top}+\Gamma(t)C(t)\big]\big[I-\Gamma(t)\bar{C}(t)\big]^{-1}\Upsilon^2(t),\ \theta\in[0,t], \ t\in[0,T].
		\end{aligned}\end{equation}
		In the meanwhile, let $\tilde{\Upsilon}^1(t)\equiv-\int_t^{T}\Lambda(t,s)\mathbb{E}^{\mathcal{F}_{t}}[\phi(t)]ds$, $t\in[T/2,T]$, we assert that $(\tilde{\Upsilon}^1(\cdot),0)$ is the solution to BSDE
		\begin{equation*}\left\{\begin{aligned}
				&d\Upsilon(t)=\Upsilon^1(t)dt+\Upsilon^2(t)dW(t),\ t\in[T/2,T],\\
				&\Upsilon(T)=0.
			\end{aligned}\right.\end{equation*}
		Thus
		\begin{equation*}
			\Upsilon(T/2)=\tilde{\Upsilon}^1(T/2)=-\int_{T/2}^T\Lambda(T/2,s)\mathbb{E}^{\mathcal{F}_{T/2}}[\phi(T/2)]ds.
		\end{equation*}
		When $t\in[0,T/2]$, designate $\tilde{\Upsilon}^2(t)=-\int_t^{t+\theta}\Lambda(t,s)\mathbb{E}^{\mathcal{F}_{t}}[\phi(t)]ds$, $\theta\in[0,t]$, the aim of the subsequent steps is to verify that $(\tilde{\Upsilon}^2(\cdot),0)$ is the solution to BSDE
		\begin{equation}\left\{\begin{aligned}\label{Gamma}
				&d\Upsilon(t)=\Upsilon^1(t)dt+\Upsilon^2(t)dW(t),\ t\in[0,T/2],\\
				&\Upsilon(T/2)=-\int_{T/2}^T\Lambda(T/2,s)\mathbb{E}^{\mathcal{F}_{T/2}}[\phi(T/2)]ds.
			\end{aligned}\right.\end{equation}
		Differentiating on $\tilde{\Upsilon}^2(t)$ with respect to $t$, we gain
		\begin{equation}\begin{aligned}\label{Gamma2}
				d\tilde{\Upsilon}^2(t)=&-\Lambda(t,t+\theta)\mathbb{E}^{\mathcal{F}_t}[\phi(t)]+\Lambda(t,t)\mathbb{E}^{\mathcal{F}_t}[\phi(t)]
-\int_t^{t+\theta}\dot{\Lambda}(t,s)\mathbb{E}^{\mathcal{F}_t}[\phi(t)]ds\\
				&-\int_t^{t+\theta}\Lambda(t,s)\Big\{A^1(t)\mathbb{E}^{\mathcal{F}_t}[\phi(t)]+B(t)\mathbb{E}^{\mathcal{F}_t}[\psi(t)]\\
                &\qquad +C(t)\mathbb{E}^{\mathcal{F}_t}[\bar{\psi}(t)]+D(t)\mathbb{E}^{\mathcal{F}_t}[u^2_t]\Big\}ds,\ \theta\in[0,t],\ t\in[0,T/2].
		\end{aligned}\end{equation}
		Noting
		\begin{equation}\label{Gamma2 tilde}
				\mathbb{E}^{\mathcal{F}_{t}}[\tilde{\Upsilon}^2(t)]=-\int_t^{t+\theta}\Lambda(t,s)ds\mathbb{E}^{\mathcal{F}_t}[\phi(t)],\ \theta\in[0,t].
		\end{equation}
		Remind (\ref{relationship---}) and (\ref{u2---}), we get
		\begin{equation}\begin{aligned}\label{eq4.23}
				&B(t)\mathbb{E}^{\mathcal{F}_{t}}[\psi(t)]+C(t)\mathbb{E}^{\mathcal{F}_t}[\bar{\psi}(t)]+D(t)\mathbb{E}^{\mathcal{F}_t}[u^2_t]\\
				=&\ B(t)\Gamma(t)\mathbb{E}^{\mathcal{F}_t}[\phi(t)]+B(t)\mathbb{E}^{\mathcal{F}_t}[\Upsilon(t)]+C(t)\big[I-\Gamma(t)\bar{C}(t)\big]^{-1}\\
                &\times\bigg\{\mathbb{E}^{\mathcal{F}_t}[\Upsilon^2(t)]+\Gamma(t)\bar{A}^1(t)\mathbb{E}^{\mathcal{F}_t}[\phi(t)]
                +\Gamma(t)C(t)^\top\mathbb{E}^{\mathcal{F}_t}[\Upsilon(t)]\bigg\}\\
				&-D(t)\big[\Xi^3(t)\big]^{-1}\bigg\{\big[D(t)^{\top}\Gamma(t)
				-G^2(t)^{\top}\big]\mathbb{E}^{\mathcal{F}_t}[\phi(t)]+D(t)^\top\mathbb{E}^{\mathcal{F}_t}[\Upsilon(t)]\bigg\},\ t\in[0,T].
		\end{aligned}\end{equation}
		By substituting (\ref{Pi-equation}), (\ref{Gamma2 tilde}) and (\ref{eq4.23}) into (\ref{Gamma2}), along with (\ref{Gamma--}), we can clearly demonstrate that $(\tilde{\Upsilon}^2(\cdot),0)$ is the solution to BSDE (\ref{Gamma}). Thus we conclude the proof.
	\end{proof}

	Now the necessary conditions of Problem \textbf{(L-EMLQ)} for solvability are derived.
	\begin{mythm}\label{thm4.2}
		Assuming that \textbf{(A1)-(A3)} hold, let $\bar{l}^i(t)=r^i(t)=0$ for $t\in [T,2T]$, $i=1,\ 2$. Suppose $\Gamma(\cdot)$ as well as $\Lambda(\cdot,\cdot+\theta)$ is the unique solution of $(\ref{L-equation})$ and $(\ref{Pi-equation})$, respectively. Assume $\bar{u}^2(\cdot)$ is the optimal control of the leader, the associated state trajectory denoted as $\bar{x}(\cdot)$ for the Problem (\textbf{L-EMLQ}). It is concluded that the optimal control of elephant memory $\bar{u}^2_{\cdot}$ takes the following state feedback form:
		\begin{equation}\begin{aligned}\label{u2-feedback}
		\bar{u}^2_t=-\big[\Xi^3(t)\big]^{-1}\bigg\{D(t)^\top\Gamma(t)-G^2(t)^\top-D(t)^\top\int_t^{(t+\theta)\wedge T}\Lambda(t,s)ds\bigg\}\mathbb{E}^{\mathcal{F}_t}[\phi(t)],&\\
a.e.\ \theta\in[0,t], \ t\in[0,T],\ \mathbb{P}\mbox{-}a.s..&
		\end{aligned}\end{equation}
		In addition, the optimal cost can be described as follows:
		\begin{equation}\begin{aligned}\label{leaderer-optimal-cost}
		\underset{u^2(\cdot)\in\,\mathcal{U}^2[0,T]}{\inf}J^2(\bar{u}^1(\cdot),u^2(\cdot))=J^2(\bar{u}^1(\cdot),\bar{u}^2(\cdot))=\mathbb{E}\big\langle x_0,\Pi^2(0)x_0+\eta^2(0)\big\rangle.
		\end{aligned}\end{equation}
	\end{mythm}
	\begin{proof}
		It is directly to check that the feedback (\ref{u2-feedback}) results from (\ref{relationship}), (\ref{relationship--}) and (\ref{u2---}). So all that proves to be shown is (\ref{leaderer-optimal-cost}). Memorize (\ref{leader state}) and (\ref{p2-equation}), applying It\^o's formula to $\langle m^2(\cdot),\bar{x}(\cdot)\rangle-\langle\eta^1(\cdot),\xi(\cdot)\rangle$, with the help of Lemmas \ref{dual operator}, \ref{dual operator two} and a series of calculations, we obtain
		\begin{equation}\begin{aligned}\label{eq4.26}
				&J^2(\bar{u}^1(\cdot),\bar{u}^2(\cdot))\\
			  =&\ \mathbb{E}\big\langle m^2(0), x_0 \big\rangle+\mathbb{E}\int_0^T\Big\langle\bar{u}^2(t),\mathbb{E}^{\mathcal{F}_t}\Big\{\big[ r^2(t)\bar{u}^2_t+\bar{b}(t)^\top m^2(t)-\bar{q}^2(t)^\top\xi(t)\big]^{\ast}\Big\}\Big\rangle dt.
		\end{aligned}\end{equation}
		Since $(\ref{u2-optimal})$ holds, we deduce
		\begin{equation}\begin{aligned}\label{eq4.27}
				&J^2(\bar{u}^1(\cdot),\bar{u}^2(\cdot))=\mathbb{E}\big\langle m^2(0),x_0\big\rangle=\mathbb{E}\big\langle x_0,\Pi^2(0)x_0+\eta^2(0)\big\rangle,
		\end{aligned}\end{equation}
		which completes the proof.
	\end{proof}
	
	Finally, we summarize the findings related to our stochastic linear-quadratic Stackelberg differential game with elephant memory and present the main theorem.
	
	\begin{mythm}\label{thm4.3}
		Let \textbf{(A1)-(A3)} hold, assume $\bar{l}^i(t)=r^i(t)=0$ for $t\in [T,2T]$, $i=1,2$. Suppose $\Gamma(\cdot)$ and $\Lambda(\cdot,\cdot+\theta)$ are the unique solutions of $(\ref{L-equation})$ and $(\ref{Pi-equation})$, respectively, then the optimal open-loop Stackleberg strategy $(\bar{u}^1(\cdot),\bar{u}^2(\cdot))$ is given by
		\begin{equation}\begin{aligned}\label{strategy-1}
				\bar{u}^1_t&=L_1^{\bar{u}^1}(t)\phi_t+L_2^{\bar{u}^1}(t)\phi(t)+L_3^{\bar{u}^1}(t,\theta)\mathbb{E}^{\mathcal{F}_t}[\phi(t)],\ a.e.\ \theta\in[0,t], \ t\in[0,T],\ \mathbb{P}\mbox{-}a.s.
		\end{aligned}\end{equation}
		\begin{equation}\begin{aligned}\label{strategy-2}
				\bar{u}^2_t&=L^{\bar{u}^2}(t,\theta)\mathbb{E}^{\mathcal{F}_t}[\phi(t)],\ a.e.\ \theta\in[0,t],\  t\in[0,T],\ \mathbb{P}\mbox{-}a.s.,
		\end{aligned}\end{equation}
		where
		\begin{equation*}
        \begin{aligned}
				L_1^{\bar{u}^1}(t):=&-\big[\Xi^1(t)\big]^{-1}\big[0,d^1(t)^\top\Pi^1(t)c^2(t)\big],\\
				L_2^{\bar{u}^1}(t):=&\big[\Xi^1(t)\big]^{-1}\Big\{-\big[b^1(t)^\top,0\big]\Gamma(t)\\
                &\qquad\qquad\quad-\big[d^1(t)^\top,0\big]\big[I-\Gamma(t)\bar{C}(t)\big]^{-1}\Gamma(t)\big[\bar{A}^1(t)+C(t)^\top\Gamma(t)\big]\Big\},\\
		\end{aligned}\end{equation*}
        \begin{equation}\label{coefficient}
        \begin{aligned}
		L_3^{\bar{u}^1}(t,\theta):=&\big[\Xi^1(t)\big]^{-1}\bigg\{d^1(t)^\top\Pi^1(t)d^2(t)\big[\Xi^3(t)\big]^{-1}\big[D(t)^\top\Gamma(t)-G^2(t)^\top\big]\\
				&\qquad\qquad\quad -\bigg[d^1(t)^\top\Pi^1(t)c^2(t)\big[\Xi^3(t)\big]^{-1}D(t)^\top-\big[b^1(t)^\top,0\big]\\
				&\qquad\qquad\quad -\big[d^1(t)^\top,0\big]\big[I-\Gamma(t)\bar{C}(t)\big]^{-1}\Gamma(t)C(t)^\top\bigg]\int_t^{(t+\theta)\wedge T}\Lambda(t,s)ds\bigg\},\\
				L^{\bar{u}^2}(t,\theta):=&-\big[\Xi^3(t)\big]^{-1}\bigg\{D(t)^\top\Gamma(t)-G^2(t)^\top-D(t)^\top\int_t^{(t+\theta)\wedge T}\Lambda(t,s)ds\bigg\}.
		\end{aligned}\end{equation}
	\end{mythm}
	\begin{proof}
		It is important to note that equation (\ref{strategy-2}) is equivalent to (\ref{u2-feedback}), so our primary focus is on proving (\ref{strategy-1}). By substituting (\ref{strategy-2}) into (\ref{u1-feedback}), it follows that
		\begin{equation}\begin{aligned}\label{eq4.31}
				\bar{u}^1_t=& -\big[\Xi^1(t)\big]^{-1}\Big\{d^1(t)^\top\Pi^1(t)c^2(t)\bar{x}_t+b^1(t)^\top\eta^1(t)\\
				&\qquad\qquad\qquad +d^1(t)^\top\bar{\eta}^1(t)+d^1(t)^\top\Pi^1(t)d^2(t)\bar{u}^2_t\Big\}\\
				=& -\big[\Xi^1(t)\big]^{-1}\bigg\{\big[0,d^1(t)^\top\Pi^1(t)c^2(t)\big]\phi_t+\big[b^1(t)^\top,0\big]\psi(t)\\
				&\qquad\qquad\qquad +\big[d^1(t)^\top,0\big]\bar{\psi}(t)-d^1(t)^\top\Pi^1(t)d^2(t)\big[\Xi^3(t)\big]^{-1}\bigg[D(t)^\top\Gamma(t)\\
				&\qquad\qquad\qquad -G^2(t)^\top-D(t)^\top\int_t^{(t+\theta)\wedge T}\Lambda(t,s)ds\bigg]\mathbb{E}^{\mathcal{F}_t}[\phi(t)]\bigg\}.\\
		\end{aligned}\end{equation}
		Furthermore, by plugging (\ref{relationship}) and (\ref{relationship---}) into (\ref{eq4.31}), we can deduce (\ref{strategy-1}), thus completing the proof.	
	\end{proof}
	
	\begin{Remark}\label{rem4.4}
		From Theorem \ref{thm4.3}, to derive the state feedback of $\bar{u}^1(\cdot)$, $\bar{u}^2(\cdot)$, we impose the assumptions \textbf{(A1)-(A3)} to the coefficients of (\ref{system equation}) and (\ref{cost}). In the one-dimensional case, we can obtain the explicit solutions to the Riccati equation (\ref{P1-equation}) and (\ref{P2-equation}) as follows:
		\begin{equation*}\begin{aligned}
				&\Pi^1(t)=g^1e^{\int_t^T(2a^1(s)+[c^1(s)]^2)}ds+\int_t^Te^{\int_t^s(2a^1(r)+[c^1(r)]^2)dr}\big[l^1(s)+[\bar{l}^1(s)]^\ast\big]ds,\\
				&\Pi^2(t)=g^2e^{\int_t^T(2a^1(s)+[c^1(s)]^2)ds}+\int_t^Te^{\int_t^s(2a^1(r)+[c^1(r)]^2)dr}\big[l^2(s)+[\bar{l}^2(s)]^\ast\big]ds.
		\end{aligned}\end{equation*}
		Therefore \textbf{(A1)-(A3)} can be simplified in the following two cases:
		\begin{enumerate}
			\item $\Pi^1(\cdot)>0,\quad \Pi^2(\cdot)\neq0.$\\
			\begin{equation*}\begin{aligned}
					&b^1(t)+c^1(t)d^1(t)=0,\quad a^2(t)+c^1(t)c^2(t)=0,\quad b^2(t)+c^1(t)d^2(t)=0,\quad a.e.\ t\in[0,T],\\
					&d^1(t)\neq0,\quad r^1(t)\neq0,\quad  \forall\ t\in[0,T].
			\end{aligned}\end{equation*}
			\item $\Pi^1(\cdot)>0,\quad \Pi^2(\cdot)=0.$\\
			\begin{equation*}\begin{aligned}
					&b^1(t)+c^1(t)d^1(t)=0,\quad a^2(t)+c^1(t)c^2(t)=0,\quad a.e.\ t\in[0,T].\\
					&d^1(t)\neq0,\quad r^1(t)=0,\quad  r^2(t)>0,\quad \forall\ t\in[0,T].
			\end{aligned}\end{equation*}
		\end{enumerate}
		
	\end{Remark}
	
\section{An applications to dynamic cooperative advertising problem}

	This section verifies the validity of the theoretical results in preceding sections. To this end, the dynamic cooperative advertising problem of a manufacturer and a retailer is transformed into a Stackelberg differential game problem with elephant memory.
	
	In this section, the attention is drawn to a modified model for a dynamic cooperative advertising problem in the marketing channel (see \cite{JSZ00}). The model under consideration comprises two participants selling some good: the manufacturer (label $m$) and the retailer (label $r$), who are regarded as the leader and the follower, respectively. The manufacturer employs the rate of advertising efforts in national media to enhance their brand image, while the retailer utilizes local promotional activities to stimulate demand among consumers. As demonstrated in the extensive existing body of literature, numerous studies have been undertaken on the impact of advertising and promotion on branding, as detailed in \cite{JTZ01, JTZ03, JZ99} and the citations therein. In particular, \cite{ZS22-1} studied the stochastic case containing asymmetric information. In this work, we consider stochastic case with elephant memory, which, to the best of our knowledge, has not been studied in the literature before.
	
	A channel is considered consisting of a manufacturer and a retailer. $u^1(\cdot)$ and $u^2(\cdot)$ indicate the rate of advertising effort of the retailer and the manufacturer, respectively. It has been determined that the manufacturer does not allocate financial resources to the retailer for any form of advertising. At the same time, the assumption is made that the profit per unit of the product is constant, denoted respectively as $\pi_m$, $\pi_r$.
	With the help of the cost functions $(\mu_m, \mu_r)$, which are positive constants, we can obtain the corresponding advertising costs $\frac{\mu_m}{2}\big[u^2(t)\big]^2$ and $\frac{\mu_r}{2}\big[u^1(t)\big]^2$ for the manufacturer and the retailer, respectively.
	
	The following discussion will introduce the brand image function, which can provide a comprehensive representation of the effectiveness of advertising activities initiated by the manufacturer, as well as the promotional activities of the retailer. This image function evolves according to the dynamic:
	\begin{equation}\left\{\begin{aligned}\label{image-function}
			\dot{p}(t)=&\ \lambda_me^{-\frac{\tau t}{2}}u^2(t)-\lambda_re^{-\frac{\tau t}{2}}u^1(t)-\delta p(t),\ t \in [0,T],\\
			      p(0)=&\ p_0,
	\end{aligned}\right.\end{equation}
	where the positive parameters $\lambda_m$ and $\lambda_r$ denote the effective proportion of the manufacturer and the retailer, respectively. $\delta \ge 0$ represents the exogenously determined rate of depreciation of commodity credits. It is evident that advertising exerts a positive influence on brand image, while promotion demonstrates an inverse effect from (\ref{image-function}).
	It is further hypothesized that the presence of both advertisements and promotions has a stimulatory effect on consumer demand, a phenomenon which is supported by the extant evidence. The relationship between the demand for goods $q(\cdot)$ and the function of image $p(\cdot)$ is as follows:
	\begin{equation*}
		q(t)=\sqrt{\sigma_m+\sigma_r}\big[p(t)\big]^2,
	\end{equation*}
	where non-negative parameters $\sigma_m$, $\sigma_r$ indicate the effectiveness of the manufacturer and the retailer advertising, respectively. Thus, the corresponding objective functionals of the manufacturer and the retailer as following:
	\begin{equation}\left\{\begin{aligned}\label{obj- func}
			J_m=&\int_0^T\bigg[-\pi_m\sqrt{\sigma_m+\sigma_r}e^{\tau t}\big[p(t)\big]^2
			+\frac{\mu_m}{2}[u^2_t]^2\bigg]dt+m^2e^{\tau T}\big[p(T)\big]^2,\\
			J_r=&\int_0^T\bigg[-\pi_r\sqrt{\sigma_m+\sigma_r}e^{\tau t}\big[p(t)\big]^2+\frac{\mu_r}{2}[d^1u^1_t]^2\bigg]dt+m^1e^{\tau T}\big[p(T)\big]^2,
		\end{aligned}\right.\end{equation}
	where positive cost coefficients $m^1, m^2$ represent the cost of maintaining a pretty brand in the terminal time, and $m^1\ge\pi_r\sqrt{\sigma_m+\sigma_r}e^{[(c^1)^2-\tau-2\delta]T}$. Let $x(t)=e^{\frac{1}{2}\tau t}p(t)$, then the above problem can be converted to a standard linear-quadratic differential game. Hence (\ref{image-function})-(\ref{obj- func}) become
	\begin{equation}\left\{\begin{aligned}\label{modify-deter}
			\dot{x}(t)=&\Big(-\frac{1}{2}\tau-\delta\Big) x(t)-\lambda_r u^1(t)+\lambda_m u^2(t),\\
			J_m=&\int_0^T\bigg[-\pi_m\sqrt{\sigma_m+\sigma_r}\big[x(t)\big]^2
			+\frac{\mu_m}{2}[u^2_t]^2\bigg]dt+m^2 \big[x(T)\big]^2,\\
			J_r=&\int_0^T\bigg[-\pi_r\sqrt{\sigma_m+\sigma_r}\big[x(t)\big]^2+\frac{\mu_r}{2}[d^1u^1_t]^2\bigg]dt+m^1\big[x(T)\big]^2.
		\end{aligned}\right.\end{equation}
	
	However, within a real market economy, it is possible that businessmen will determine the timing of advertisements based on the relationship between prior advertisements and sales of products. Consequently, it is hypothesized that both the rate of advertisements $u^1(\cdot)$ and $u^2(\cdot)$ placed by the manufacturer and the retailer depend on the entire past history, while the elephant memory of the state exerts an influence on the cost functional of each constituent of the channel. The factors influencing the cost functionals of the manufacturer and the retailer are not uniquely determined, so the model is extended to the stochastic case with elephant memory:
	\begin{equation}\left\{\begin{aligned}\label{modify-stochas}
			dx(t)=&\bigg[\Big(-\frac{1}{2}\tau-\delta\Big) x(t)      -\lambda_re^{-\tau_1}u^1_t+\lambda_me^{-\tau_2}u^2_t\bigg]dt\\
			&+\Big[c^1x(t)+d^1u^1_t+d^2u^2_t\Big]dW(t),\ t\in[0,T],\\
			J_m=&\ \mathbb{E}\bigg\{\int_0^T\bigg[-\pi_m\sqrt{\sigma_m+\sigma_r}\big[x(t)\big]^2-\bar{l}^2x^2_t+\frac{\mu_m}{2}e^{-\tau_2}[u^2_t]^2\bigg]dt+m^2\big[x(T)\big]^2\bigg\},\\
			J_r=&\ \mathbb{E}\bigg\{\int_0^T\bigg[-\pi_r\sqrt{\sigma_m+\sigma_r}\big[x(t)\big]^2-\bar{l}^1x^2_t+\frac{\mu_r}{2}e^{-\tau_1}[d^1u^1_t]^2\bigg]dt+m^1\big[x(T)\big]^2\bigg\},
		\end{aligned}\right.\end{equation}
	where positive parameters $\tau_1, \tau_2$ are given constants, denoted as profit margin per unit of advertising. $c^1,d^1,d^2$ are all deterministic constants satisfying
	\begin{equation*}
		d^1\ne0, \quad (c^1)^2-\tau-2\delta\ge0, \quad   b^1+c^1d^1=0,\quad b^2+c^1d^2\neq0.
	\end{equation*}
	
	Now we apply the theoretical results in the above two sections to seek the optimal open-loop Stackelberg strategy. The coefficients take the following values:
	\begin{equation*}\begin{aligned}
			&a^1(t)=-\frac{1}{2}\tau-\delta,\quad b^1(t)=-\lambda_re^{-\tau_1},\quad b^2(t)=\lambda_me^{-\tau_2},\quad a^2(t)=c^2(t)=0,\\
			&l^1(t)=-\pi_r\sqrt{\sigma_m+\sigma_r}, \quad \bar{l}^1(t)=-\bar{l}^1, \quad  r^1(t)=\frac{\mu_r}{2}e^{-\tau_1}[d^1]^2,\quad g^1=m^1,\\
			&l^2(t)=-\pi_m\sqrt{\sigma_m+\sigma_r},\quad \bar{l}^2(t)=-\bar{l}^2, \quad  r^2(t)=\frac{\mu_m}{2}e^{-\tau_2},\quad g^2=m^2.
	\end{aligned}\end{equation*}
	
	Apparently the coefficients satisfies \textbf{(A1)-(A3)}, now
	\begin{equation*}
		\Xi^1(t)=\bar{\Pi}(t)\big[d^1\big]^2,\quad \Xi^2(t)=\Xi^3(t)=\frac{\mu_m}{2}e^{-\tau_2},
	\end{equation*}
	where $\bar{\Pi}(\cdot)=\Pi^1(\cdot)+\frac{\mu_r}{2}e^{-\tau_1}$, and $\Pi^1(\cdot)$ is the solution the following linear ODE:
	\begin{equation}\left\{\begin{aligned}\label{ex2-P1}
			\dot{\Pi}^1(t)=&\Big[\tau+2\delta-(c^1)^2\Big]\Pi^1(t)+\pi_r\sqrt{\sigma_m+\sigma_r},\ t\in[0,T],\\
			\Pi^1(T)=&\ m^1.
	\end{aligned}\right.\end{equation}
	It is easy to verify that $\Pi^1(t)=-\pi_r\sqrt{\sigma_m+\sigma_r}\int_{t}^{T}e^{[(c^1)^2-\tau-2\delta](s-t)}ds+m^1>0$, hence $\Xi^1>0$ and $\big[\Xi^1(t)\big]^{-1}=\big[\bar{\Pi}(t)\big]^{-1}[d^1]^{-2}$. Now (\ref{P2-equation}) reads
	\begin{equation}\left\{\begin{aligned}\label{ex2-P2}
			\dot{\Pi}^2(t)=&\Big[\tau+2\delta-(c^1)^2\Big]\Pi^2(t)+\pi_m\sqrt{\sigma_m+\sigma_r},\ t\in[0,T],\\
			\Pi^2(T)=&\ m^2,
		\end{aligned}\right.\end{equation}
	and it follows that $\Pi^2(t)=-\pi_m\sqrt{\sigma_m+\sigma_r}\int_t^Te^{[(c^1)^2-\tau-2\delta](s-t)}ds+m^2$. In this case,
	\begin{eqnarray*}\begin{aligned}
			&A^2 \equiv \bar{A}^2\equiv\bar{D}\equiv G^1\equiv H \equiv0,\quad
			A^1(t)=-\Big(\frac{1}{2}\tau+\delta\Big)\begin{bmatrix}
				1 & 0\\
				0 & 1
			\end{bmatrix},\\
            &B(t)=\big[c^1\big]^2\bar{\Pi}^{-1}(t)\begin{bmatrix}
				\bar{\Pi}(t)\Pi^2(t) & -1\\
				-1 & 0
			\end{bmatrix},\quad
            C(t)=c^1\bar{\Pi}^{-1}(t)\begin{bmatrix}
				-\bar{\Pi}^{-1}(t)\Pi^2(t) & 1\\
				1 & 0
			\end{bmatrix},\\
			&\bar{A}^1(t)= \begin{bmatrix}
				c^1 & 0\\
				0 & c^1
			\end{bmatrix},\quad
            \bar{C}(t)=\bar{\Pi}^{-1}(t)\begin{bmatrix}
				\bar{\Pi}^{-1}(t)\Pi^2(t) & -1\\
				-1 & 0
			\end{bmatrix},\\
			&D(t)= \begin{bmatrix}
				0 \\
				b^2+c^1d^2\Pi^1(t)\bar{\Pi}^{-1}(t)
			\end{bmatrix},\quad
            G^2(t)= \begin{bmatrix}
				-\Pi^1(t)(b^2+c^1d^2) \\
				0
			\end{bmatrix}.
	\end{aligned}\end{eqnarray*}
	Denote
	\begin{equation*}\begin{aligned}
			\Gamma(t):= \begin{bmatrix}
				\Gamma^1(t) & \Gamma^2(t)\\
				\Gamma^2(t) & \Gamma^3(t)
			\end{bmatrix},\quad
            \Lambda(t,\theta):= \begin{bmatrix}
				\Lambda^1(t,\theta) & \Lambda^2(t,\theta)\\
				\Lambda^2(t,\theta) & \Lambda^3(t,\theta)
			\end{bmatrix}.
	\end{aligned}\end{equation*}
	With some computations, we can get for $t\in[0,T]$,
	\begin{equation*}\begin{aligned}
			&[I-\Gamma(t)\bar{C}(t)]^{-1}\\
			=& \frac{1}{\big[\Gamma^2(t)+\Pi^1(t)\big]^2-\Gamma^1(t)\Gamma^3(t)-\Gamma^2(t)\Pi^2(t)}\\
            &\times\begin{bmatrix}
				\Gamma^2(t)\bar{\Pi}(t)+\big[\bar{\Pi}(t)\big]^2 & -\Gamma^1(t)\bar{\Pi}(t)\\
				\Gamma^2(t)\Pi^2(t)-\Gamma^3(t)\bar{\Pi}(t) & \big[\bar{\Pi}(t)\big]^2+\Gamma^2(t)\bar{\Pi}(t)-\Gamma^1(t)\Pi^2(t)
			\end{bmatrix}.
	\end{aligned}\end{equation*}
	It follows that for $t\in[0,T]$ (omitting $t$)
	\begin{equation*}\begin{aligned}
			\Omega^1(t)=& \big[c^1\big]^2\bar{\Pi}^{-1}\begin{bmatrix}
				\Gamma^1\bar{\Pi}\Pi^2-\Gamma^2 & \Gamma^2\bar{\Pi}\Pi^2-\Gamma^3\\
				-\Gamma^1 & -\Gamma^2
			\end{bmatrix}+\frac{\big[c^1\big]^2\bar{\Pi}^{-1}}{\big[\Gamma^2+\Pi^1\big]^2-\Gamma^1\Gamma^3-\Gamma^2\Pi^2}\\
            &\times\begin{bmatrix}
				-\Gamma^3-\Pi^2 & \Gamma^2+\bar{\Pi}\\
				\Gamma^2+\bar{\Pi} & -\Gamma^1
			\end{bmatrix}\begin{bmatrix}
				\Gamma^1 & \Gamma^2\\
				\Gamma^2 & \Gamma^3
			\end{bmatrix}\begin{bmatrix}
				\Gamma^2+\bar{\Pi}-\Gamma^1\bar{\Pi}^{-1}\Pi^2 & \Gamma^3-\Gamma^2\bar{\Pi}^{-1}\Pi^2\\
				\Gamma^1 & \Gamma^2+\bar{\Pi}
			\end{bmatrix}\\
       &+\frac{2}{\mu_m}e^{\tau_2}\big[b^2+c^1d^2\Pi^1\bar{\Pi}^{-1}\big]^2\begin{bmatrix}
				0 & 0\\
				-\frac{\Pi^1[b^2+c^1d^1]}{b^2+c^1d^2\Pi^1\bar{\Pi}^{-1}}-\Gamma^2 & -\Gamma^3
			\end{bmatrix},
	\end{aligned}\end{equation*}	
	\begin{equation*}\begin{aligned}
			\Omega^2(t)=&\big[c^1\big]^2\bar{\Pi}^{-1}\begin{bmatrix}
				-\bar{\Pi}\Pi^2 & 1\\
				1 & 0
			\end{bmatrix}-\frac{\big[c^1\big]^2\bar{\Pi}^{-1}}{\big[\Gamma^2+\Pi^1\big]^2-\Gamma^1\Gamma^3-\Gamma^2\Pi^2} \begin{bmatrix}
				-\Gamma^3-\Pi^2 & \Gamma^2+\bar{\Pi}\\
				\Gamma^2+\bar{\Pi} & -\Gamma^1\\
			\end{bmatrix}\\
			&\times
			\begin{bmatrix}
				\Gamma^2-\Gamma^1\bar{\Pi}\Pi^2 & \Gamma^1\\
				\Gamma^3-\Gamma^2\bar{\Pi}\Pi^2 & \Gamma^2
			\end{bmatrix}+\frac{2}{\mu_m}e^{\tau_2}\big[b^2+c^1d^2\Pi^1\bar{\Pi}^{-1}\big]\begin{bmatrix}
				0 & 0\\
				0 & 1
			\end{bmatrix},
	\end{aligned}\end{equation*}
	\begin{equation*}\begin{aligned}
			\Omega^3(t)=&\big[c^1\big]^2\bar{\Pi}^{-1}\begin{bmatrix}
				2\Gamma^1\Gamma^2-\big[\Gamma^1\big]^2\bar{\Pi}\Pi^2 & \big[\Gamma^2 \big]^2+\Gamma^1\Gamma^3-\Gamma^1\Gamma^2\bar{\Pi}\Pi^2\\
				\big[\Gamma^2\big]^2+\Gamma^1\Gamma^3-\Gamma^1\Gamma^2\bar{\Pi}\Pi^2 & 2\Gamma^2\Gamma^3-\big[\Gamma^2\big]^2\bar{\Pi}\Pi^2
			\end{bmatrix}\\
			&+\frac{2}{\mu_m}e^{\tau_2}\big[b^2+c^1d^2\Pi^1\bar{\Pi}^{-1}\big]^2\begin{bmatrix}
				\Big(\Gamma^2+\frac{\Pi^1[b^2+c^1d^2]}{b^2+c^1d^2\Pi^1\bar{\Pi}^{-1}}\Big)^2 & \Gamma^2\Gamma^3+\frac{\Gamma^3\Pi^1[b^2+c^1d^2]}{b^2+c^1d^2\Pi^1\bar{\Pi}^{-1}}\\
				\Gamma^2\Gamma^3+\frac{\Gamma^3\Pi^1[b^2+c^1d^2]}{b^2+c^1d^2\Pi^1\bar{\Pi}^{-1}} & \big[\Gamma^3\big]^2
			\end{bmatrix}\\
			&-\frac{\big[c^1\big]^2\bar{\Pi}^{-1}}{\big[\Gamma^2+\Pi^1\big]^2-\Gamma^1\Gamma^3-\Gamma^2\Pi^2}\begin{bmatrix}
				\Gamma^2+\bar{\Pi}-\Gamma^1\bar{\Pi}^{-1}\Pi^2 & \Gamma^1\\
				\Gamma^3-\Gamma^2\bar{\Pi}^{-1}\Pi^2 & \Gamma^2+\bar{\Pi}
			\end{bmatrix}\\
			&\times\begin{bmatrix}
				\Gamma^2\bar{\Pi}+\big[\bar{\Pi}\big]^2 & -\Gamma^1\bar{\Pi}\\
				\Gamma^2\Pi^2-\Gamma^3\bar{\Pi} & \big[\bar{\Pi}\big]^2+\Gamma^2\bar{\Pi}-\Gamma^1\Pi^2
			\end{bmatrix}\begin{bmatrix}
				\Gamma^1 & \Gamma^2\\
				\Gamma^2 & \Gamma^3
			\end{bmatrix}\\
            &\times\begin{bmatrix}
				\Gamma^2+\bar{\Pi}-\Gamma^1\bar{\Pi}^{-1}\Pi^2 & \Gamma^3-\Gamma^2\bar{\Pi}^{-1}\Pi^2\\
				\Gamma^1 & \Gamma^2+\bar{\Pi}
			\end{bmatrix}.
	\end{aligned}\end{equation*}
	Then (\ref{L-equation}) becomes:
	\begin{equation}\left\{\begin{aligned}\label{ex2-L}
			\dot{\Gamma}(t)&=(\tau+2\delta)\Gamma(t)+\Lambda(t,t+\theta)I_{[0,T]}(t+\theta),\ t\in[0,T],\ \theta \in[0,t],\\
			\Gamma(T)&=0,
		\end{aligned}\right.\end{equation}
	and $\Lambda(t,t+\theta)$ satisfies
	\begin{equation}\left\{\begin{aligned}\label{ex2-Pi}
			\frac{\partial\Lambda(t,s)}{\partial t}=&\ (\tau+2\delta)\Lambda(t,s),\\
			\Lambda(t,t)=&\left(\int_t^{(t+\theta)\wedge T}\Lambda(t,s)ds\right)\Omega^1(t)+\Omega^1(t)^\top\int_t^{(t+\theta)\wedge T}\Lambda(t,s)ds\\
			&+\left(\int_t^{(t+\theta)\wedge T}\Lambda(t,s)ds\right)\Omega^2(t)\int_t^{(t+\theta)\wedge T}\Lambda(t,s)ds+\Omega^3(t).
		\end{aligned}\right.\end{equation}
	The difficulty of solving the matrix-valued equation (\ref{ex2-L}) is compounded by the difficulty of determining the uniqueness of the solution to equation (\ref{ex2-Pi}), and by the fact that $\Lambda^1$, $\Lambda^2$ and $\Lambda^3$ are coupled in the matrix-valued equation (\ref{ex2-L}), which makes it even more difficult to give an analytical expression for (\ref{ex2-L}).
	
	Finally, the optimal open-loop Stackelberg strategy is as follows:
	\begin{equation*}
			\bar{u}^1_t=L(t)\phi(t)+L^1(t,\theta)\mathbb{E}^{\mathcal{F}_t}[\phi(t)],\ a.e.\ t\in[0,T],\ \mathbb{P}\mbox{-}a.s.
	\end{equation*}
	\begin{equation*}
			\bar{u}^2_t=L^2(t,\theta)\mathbb{E}^{\mathcal{F}_t}[\phi(t)],\ a.e.\ t\in[0,T],\ \mathbb{P}\mbox{-}a.s.,
	\end{equation*}
	where
	\begin{equation*}\begin{aligned}
      L(t):=&\big[\Xi^1(t)\big]^{-1}\bigg\{-b^1[\Gamma^1(t), \Gamma^2(t)]- \frac{c^1\bar{\Pi}^{-1}(t)}{\big[\Gamma^2(t)+\Pi^1(t)\big]^2-\Gamma^1(t)\Gamma^3(t)-\Gamma^2(t)\Pi^2(t)}                           \bigg\}[d^1,0]\\
      &\times \begin{bmatrix}
				\Gamma^2(t)\bar{\Pi}(t)+\big[\bar{\Pi}(t)\big]^2 & -\Gamma^1(t)\bar{\Pi}(t)\\
				\Gamma^2(t)\Pi^2(t)-\Gamma^3(t)\bar{\Pi}(t) & \big[\bar{\Pi}(t)\big]^2+\Gamma^2(t)\bar{\Pi}(t)-\Gamma^1(t)\Pi^2(t)
			\end{bmatrix}
				\begin{bmatrix}
					\Gamma^1(t) & \Gamma^2(t)\\
					\Gamma^2(t) & \Gamma^3(t)
				\end{bmatrix}\\
				&\times\begin{bmatrix}
					\Gamma^2(t)+\bar{\Pi}(t)-\Gamma^1(t)\bar{\Pi}^{-1}(t)\Pi^2(t) & \Gamma^3(t)-\Gamma^2(t)\bar{\Pi}^{-1}(t)\Pi^2(t)\\
					\Gamma^1(t) & \Gamma^2(t)+d^1
				\end{bmatrix},\\
		L^1(t,\theta):=&\big[\Xi^1(t)\big]^{-1}\bigg\{\frac{2}{\mu_m}e^{\tau_2}d^1d^2\Pi^1(t)\big[b^2+c^1d^2\Pi^1(t)\bar{\Pi}^{-1}(t)\big]\\
        &\times \bigg[\Gamma^2(t)+\frac{\Pi^1(t)[b^2+c^1d^2]}{b^2+c^1d^2\Pi^1(t)\bar{\Pi}^{-1}(t)}, \Gamma^3(t)\bigg]\\
		&\times \bigg\{[-b,0]- \frac{[d^1,0]\bar{\Pi}^{-1}(t)}{\big[\Gamma^2(t)+\Pi^1(t)\big]^2-\Gamma^1(t)\Gamma^3(t)-\Gamma^2(t)\Pi^2(t)}\\
		&\times \begin{bmatrix}
			\Gamma^2(t)\bar{\Pi}(t)+\big[\bar{\Pi}(t)\big]^2 & -\Gamma^1(t)\bar{\Pi}(t)\\
			\Gamma^2(t)\Pi^2(t)-\Gamma^3(t)\bar{\Pi}(t) & \big[\bar{\Pi}(t)\big]^2+\Gamma^2(t)\bar{\Pi}(t)-\Gamma^1(t)\Pi^2(t)
		\end{bmatrix}\\
		&\times \begin{bmatrix}
			\Gamma^1(t)\bar{\Pi}^{-1}(t)\Pi^2(t)-\Gamma^2(t) & -\Gamma^1(t)\\
			\Gamma^2(t)\bar{\Pi}^{-1}(t)\Pi^2(t)-\Gamma^3(t)& -\Gamma^2(t)
		\end{bmatrix}\bigg\} \times \int_t^{(t+\theta)\wedge T}\Lambda(t,s)ds\bigg\},\\
			L^2(t,\theta):=&-\frac{2}{\mu_m}e^{\tau_2}\big[b^2+c^1d^2\Pi^1(t)\bar{\Pi}^{-1}(t)\big]\begin{bmatrix}
				\Gamma^2(t)+\frac{\Pi^1(t)[b^2+c^1d^2]}{b^2+c^1d^2\Pi^1(t)\bar{\Pi}^{-1}(t)}-\int_t^{(t+\theta)\wedge T}\Lambda^2(t,s)ds\\
				\Gamma^3(t)-\int_t^{(t+\theta)\wedge T}\Lambda^3(t,s)ds
			\end{bmatrix}^\top.
	\end{aligned}\end{equation*}
	
	Let $c^1=1$, $m^1=1000$, $m^2=2000$, $\pi_m=0.002$, $\pi_r=0.1$, $\sigma_m=0.3$, $\sigma_r=0.7$, $\delta=0.3$, $\tau=0.2$, $T=10$, we show the following figures of (\ref{ex2-P1}) and (\ref{ex2-P2}), respectively:
	\begin{figure}[H]
		\centering
		\includegraphics[height=7cm,width=12cm]{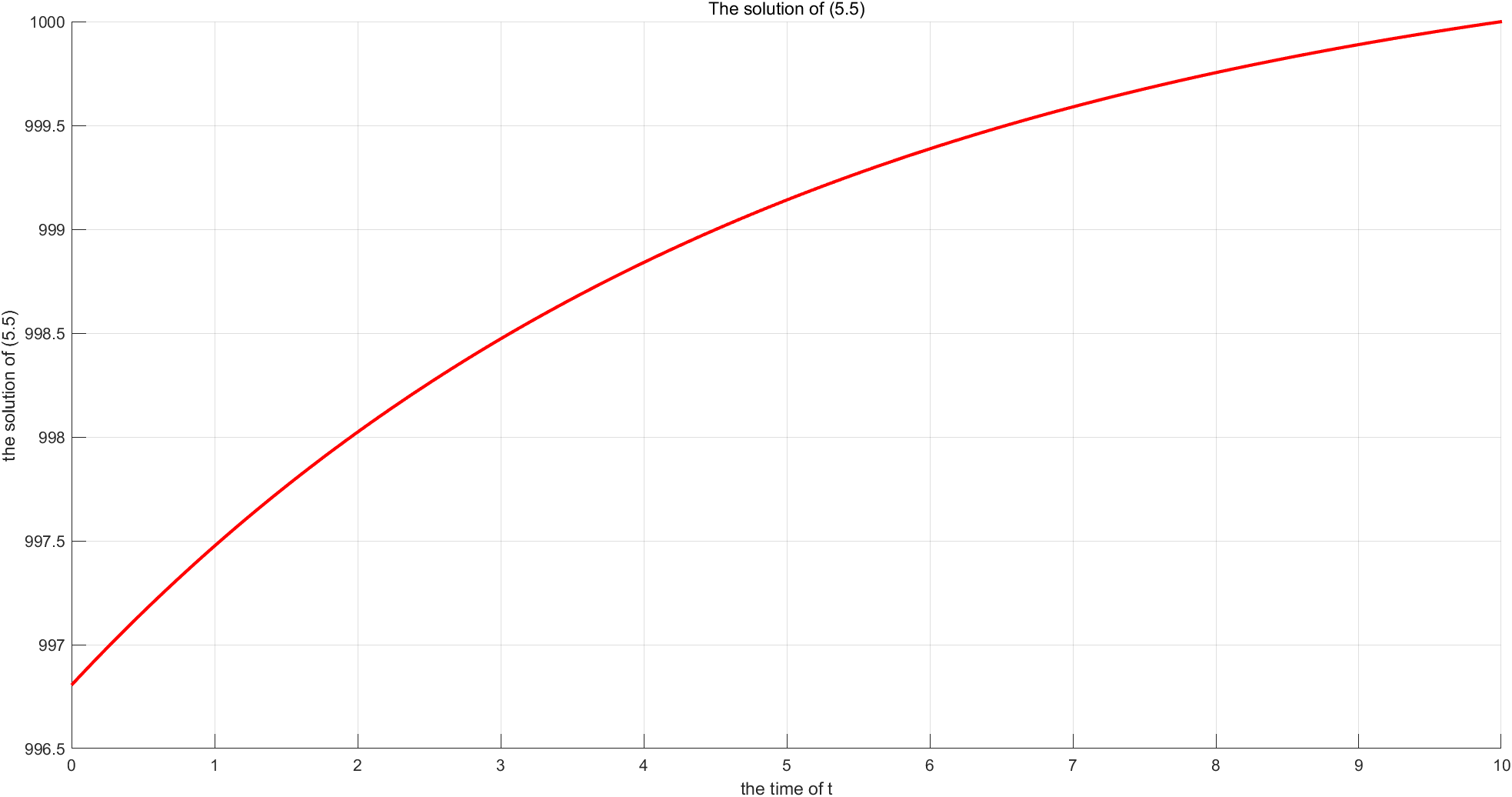}
		\caption{The solution to linear ODE (\ref{ex2-P1})}
	\end{figure}
	\begin{figure}[H]
		\centering
		\includegraphics[height=7cm,width=12cm]{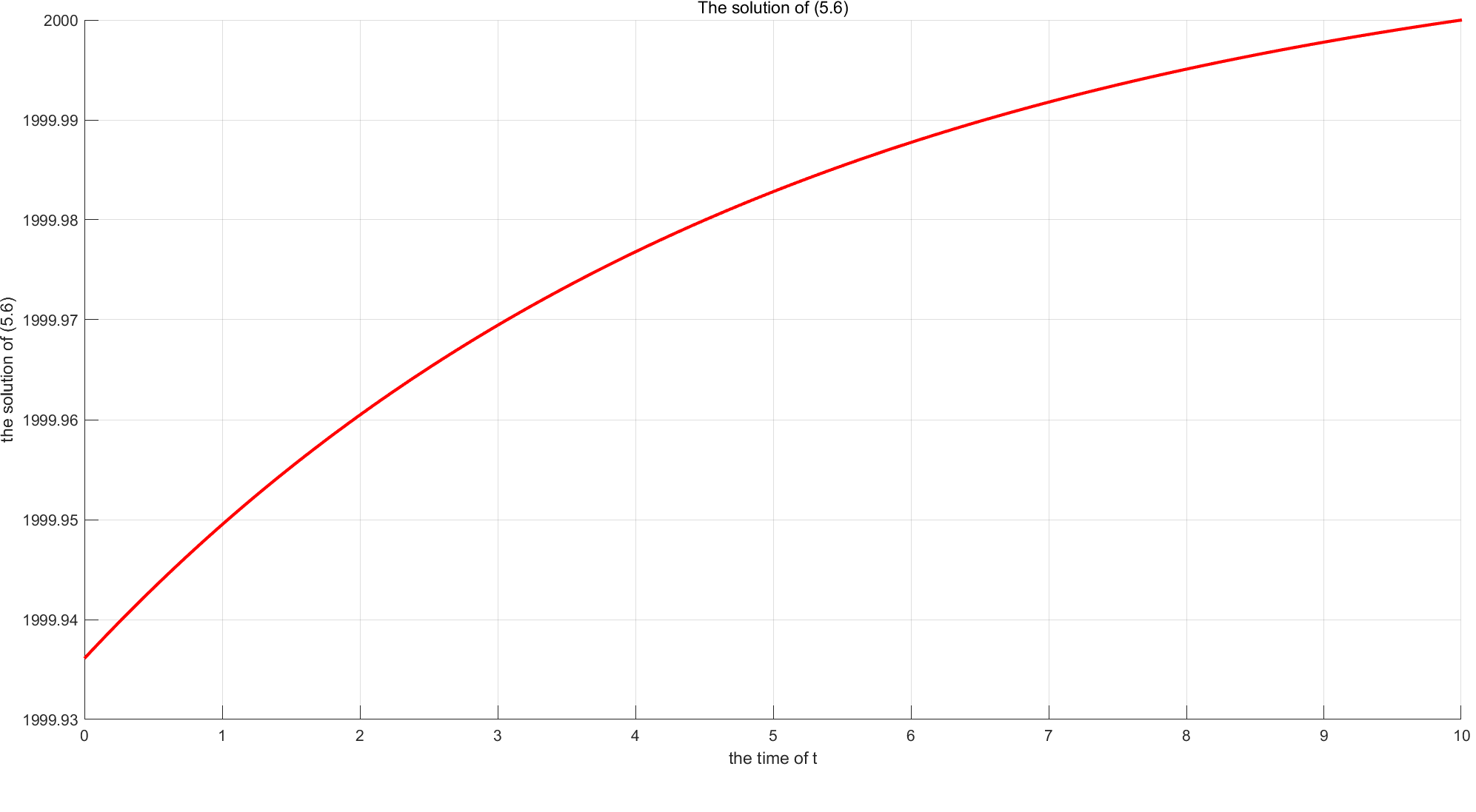}
		\caption{The solution to linear ODE (\ref{ex2-P2})}
	\end{figure}
	
\section{Concluding remarks}
	
	The work in this paper resembles that in \cite{MS22}. Our model is general in that the state equation for both the leader and the follower includes the state of elephant memory and the control of elephant memory, which are part of the diffusion term. However, the cost functionals (\ref{cost}) don't contain $\big[u^i(t)\big]^2, i=1, 2$, due to the requirement of linearity imposed by the operator in Lemma (\ref{dual operator}). To start, we address the optimization problem for the follower, which is a stochastic linear-quadratic optimal control problem with elephant memory for any choice selected by the leader. We present the necessary and sufficient conditions for solving the follower's problem by introducing a Riccati equation (\ref{P1-equation}) (see Theorem \ref{thm3.2}).
	
	Subsequently, we examine the optimization problem for the leader, which is a stochastic linear-quadratic optimal control problem characterized by a state equation formed by an SDE with elephant memory and an ABSDE. In order to solve for the optimal control of the follower and the leader, it is necessary to impose two additional unnatural conditions, \textbf{(A1)} and \textbf{(A2)}, which are more complex than those presented in \cite{MS22}. This is due to the fact that, at this stage, it is possible to access only the Fr\'echet derivative of the continuous linear operator and the corresponding dual theory as we know.
	
	The forward and backward variables obtained during the optimization of both the follower and the leader are stacked together. Alongside this, a special matrix-valued equation (\ref{L-equation}) is introduced. Ultimately, we derive the necessary conditions for the solvability of the leader's problem (see Theorem \ref{thm4.2}). In summary, the open-loop Stackelberg strategy is shown in Theorem \ref{thm4.3}. At the end of this paper, theoretical results is illustrated by means of an example concerning a dynamic advertising problem with elephant memory.
	
	As the Fr\'echet derivative and dual theory of higher-order operator are established, it may become possible to deal with more general cases, in line with the weakening or removal of the assumptions \textbf{(A1)-(A3)}. It is essential to propose and carefully investigate some reasonable relations among the adjoint variables and the state variable. Furthermore, it is a fascinating exercise to provide analytical expressions for the Riccati equations and to investigate the solvability of the related matrix-valued equations. Nevertheless, these subjects are intricate and exacting.  We will consider them in the near future.

\end{document}